\documentclass[11pt, oneside]{article}

\usepackage{geometry}                		
\usepackage{amssymb}
\usepackage{amsthm}
\usepackage{amsmath}

\usepackage{tikz}
\usetikzlibrary{calc}

\usepackage{verbatim}

\newtheorem{thm}{Theorem}[section]
\newtheorem{lemma}[thm]{Lemma}

\newtheorem{cor}[thm]{Corollary}

\makeatletter
\@addtoreset{case}{thm}
\makeatletter

\theoremstyle{definition}

\def\F {{\mathcal F}}
\def\rt {{r_{4}}}
\def\rf {{r_{5}}}

\def\ef {{\epsilon_{5}}}

\title{Size of the Largest Induced Forest in Subcubic Graphs of Girth at least Four and Five}

\author{Tom Kelly\thanks{Department of Combinatorics and Optimization, University of Waterloo, Waterloo, Ontario, Canada. Email: \texttt{t9kelly@uwaterloo.ca}}
\and
Chun-Hung Liu\thanks{Department of Mathematics, Princeton University, Princeton, New Jersey, USA. Email: \texttt{chliu@math.princeton.edu}. Partially supported by NSF under grant DMS-1664593.}
}

\date{\today}

\begin{document}
\maketitle

\begin{abstract}
In this paper, we address the maximum number of vertices of induced forests in subcubic graphs with girth at least four or five.
We provide a unified approach to prove that every 2-connected subcubic graph on $n$ vertices and $m$ edges with girth at least four or five, respectively, has an induced forest on at least $n-\frac{2}{9}m$ or $n-\frac{1}{5}m$ vertices, respectively, except for finitely many exceptional graphs.
Our results improve a result of Liu and Zhao and are tight in the sense that the bounds are attained by infinitely many 2-connected graphs.
Equivalently, we prove that such graphs admit feedback vertex sets with size at most $\frac{2}{9}m$ or $\frac{1}{5}m$, respectively.
Those exceptional graphs will be explicitly constructed, and our result can be easily modified to drop the 2-connectivity requirement.
\end{abstract}

\section{Introduction}
All graphs are simple in this paper.
We say that a graph $G$ is {\em subcubic} if every vertex of $G$ has degree at most three.
A graph is {\em cubic} if every vertex has degree three.

The largest size of an induced subgraph with certain properties in subcubic graphs has been extensively studied. 
An immediate corollary of Brooks' theorem is that every subcubic graph on $n$ vertices other than $K_4$ contains an independent set of size at least $\frac{n}{3}$ and an induced bipartite subgraph of size at least $\frac{2n}{3}$.
These easy upper bounds are attained by any subcubic graph whose vertices can be partitioned into triangles and hence tight.
However, the upper bounds on the size of those induced subgraphs can be significantly improved when triangles are excluded.
Staton \cite{s} proved that every triangle-free subcubic graph on $n$ vertices has an independent set with size least $\frac{5n}{14}$, improving the solution of Fajtlowicz \cite{f} of a conjecture of Albertson, Bollobas and Tucker \cite{abt}.
It was further strengthened by Zhu \cite{z} that all but finitely many triangle-free subcubic graphs on $n$ vertices have an induced bipartite subgraph on at least $\frac{5n}{7}$ vertices.
In addition, it was proved by the second author and Yu \cite{ly} that every subcubic graph other than $K_{3,3}$ can be partitioned into four independent sets such that each subgraph induced by any two parts is a disjoint union of paths.
It implies that every subcubic graph on $n$ vertices contains an induced subgraph that is a linear forest on at least $\frac{n}{2}$ vertices.

In this paper, we investigate the largest size of an induced subgraph that is a forest, which is a natural class between induced bipartite subgraphs and induced linear forests and is closely related to feedback vertex sets, which have been extensively studied.
An {\em induced forest} in a graph $G$ is an induced subgraph that is a forest, and we denote the largest size of an induced forest in $G$ by $a(G)$; a subset $S$ of vertices is a {\em feedback vertex set} of $G$ if $S$ intersects all cycles in $G$, and we denote the minimum size of a feedback vertex set of $G$ by $\phi(G)$.
It is easy to see that $a(G)+\phi(G)=\lvert V(G) \rvert$.
We remark that deciding the size of a minimum feedback vertex set of an input graph $G$ is NP-hard in general \cite{k} but is polynomial time solvable when $G$ is subcubic \cite{ukg}.

We will address the largest size of induced forests in subcubic graphs in this paper.
Study of this parameter on subcubic graphs seems crucial.
Problems on subcubic graphs are already nontrivial and have received much attention, such as Staton and Zhu's results mentioned earlier. 
In addition, for some conjectures about the size of the largest induced forests, such as a conjecture proposed by Kowalik, Lu\v{z}ar and \v{S}krekovski \cite{KLS10} on planar graphs with girth at least five, the vertex-set of every known graph attaining the conjectured bound can be partitioned into induced subcubic subgraphs.
In fact, the machinery developed in this paper will be used in one of our follow-up papers \cite{kl} about induced forests in planar graphs.

The study of the largest induced forest in subcubic graphs can be dated to a result of Bondy, Hopkins and Staton \cite{BHS87}.

\begin{thm}\label{bhs}\cite{BHS87}
If $G$ is a connected subcubic graph on $n>4$ vertices, then $a(G)\geq\frac{5n}{8} - \frac{1}{4}$.  
If $G$ is a connected triangle-free subcubic graph, then $a(G)\geq \frac{2n}{3} - \frac{1}{3}$.
\end{thm}

The first bound in Theorem \ref{bhs} is attained by every cubic graph that can be obtained from a disjoint union of copies of $K_3$ or $K_4^+$ by adding edges so that each added edge is a cut-edge.
In this paper, $K_4^+$ denotes the graph obtained from $K_4$ by subdividing an edge. 
And the bound of the second statement in Theorem \ref{bhs} is attained by $Q_3$ and $V_8$, where $Q_3$ is the 3-dimensional cube and $V_8$ is the non-planar cubic graph on eight vertices with no $K_5$-minor.
See Figure \ref{girth four F_{2,2}}.

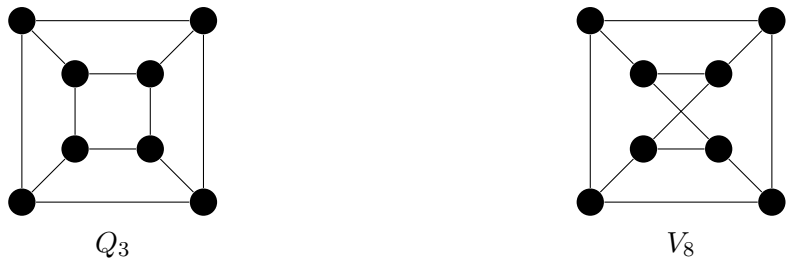
\begin{figure}
\begin{minipage}[b]{.5\linewidth}
	\centering
	\begin{tikzpicture}
		\tikzstyle{every node}=[fill=black, circle]
		\node (a) at (0, 0) {}; \node (1) at ($(a) + (135:1)$) {};
		\node (b) at (1, 0) {}; \node (2) at ($(b) + (45:1)$) {};
		\node (c) at (1, -1) {}; \node (3) at ($(c) + (-45:1)$) {};
		\node (d) at (0, -1) {}; \node (4) at ($(d) + (-135:1)$) {};
		
		\draw (1) -- (2) -- (3) -- (4) -- (1);
		\draw (a) -- (b) -- (c) -- (d) -- (a);
		
		\draw (1) -- (a);
		\draw (2) -- (b);
		\draw (3) -- (c);
		\draw (4) -- (d);
	\end{tikzpicture}\\$Q_3$
\end{minipage}%
\begin{minipage}[b]{.5\linewidth}
	\centering
	\begin{tikzpicture}
		\tikzstyle{every node}=[fill=black, circle]
		\node (a) at (0, 0) {}; \node (1) at ($(a) + (135:1)$) {};
		\node (b) at (1, 0) {}; \node (2) at ($(b) + (45:1)$) {};
		\node (c) at (1, -1) {}; \node (3) at ($(c) + (-45:1)$) {};
		\node (d) at (0, -1) {}; \node (4) at ($(d) + (-135:1)$) {};
		
		\draw (1) -- (2) -- (3) -- (4) -- (1);
		\draw (a) -- (b) -- (d) -- (c) -- (a);
		
		\draw (1) -- (a);
		\draw (2) -- (b);
		\draw (3) -- (c);
		\draw (4) -- (d);
	\end{tikzpicture}\\$V_8$
\end{minipage}
\caption{Two cubic graphs with girth four.}
\label{girth four F_{2,2}}
\end{figure}
  
Zheng and Lu \cite{ZL90} improved the second statement of Theorem \ref{bhs} for cubic graphs by showing that the additive constant can be removed for all triangle-free cubic graphs except for $Q_3$ and $V_8$.
Their result is tight for infinitely many 2-connected graphs, by taking disjoint copies of $L$ and adding edges.
In this paper, $L$ denotes the graph obtained from $K_4$ by subdividing a perfect matching once.
Their result was further improved by Liu and Zhao \cite{LZ96} as follows.

\begin{thm}\label{lz}\cite{LZ96}
If $G$ is a connected cubic graph on $n$ vertices of girth at least $g$, not $K_4,Q_3,$ or $V_8$, and if $G$ cannot be obtained from a disjoint union of copies of $K_3$ or $K_4^+$ by adding edges so that each added edge is a cut-edge, then $a(G) \geq \frac{3g-4}{4g-4}n-\frac{g-3}{2g-2}$. 
\end{thm}

Theorem \ref{lz} is tight for $g \in \{3,4\}$, but we will show that it is not tight for $g=5$.

It is not hard to see that if $G$ is a graph such that no cycle is a component, then there exists an induced forest of $G$ containing all vertices of degree less than three.
So it is reasonable to expect that the size of the largest induced forest is larger when the graph is not cubic.
Alon, Mubayi and Thomas \cite{AMT01} proved the following. 

\begin{thm}\label{subcubic girth 3}\cite{AMT01}
If $G$ is a subcubic graph on $n$ vertices and $m$ edges and $G \neq K_4$, then $a(G) \geq n-\frac{m}{4}-\frac{c}{4}$, where $c$ is the number of components of $G$ that are cubic graphs and can be obtained from a disjoint union of copies of $K_3$ or $K_4^+$ by adding edges so that each added edge is a cut-edge.
\end{thm}

Note that $m \leq \frac{3n}{2}$ for subcubic graph, so Theorem \ref{subcubic girth 3} implies the case $g=3$ in Theorem \ref{lz} and hence implies Theorem \ref{bhs}, and it also shows that the bound in the case $g=3$ of Theorem \ref{lz} is only attained by cubic graphs.

We remark that the structure of subcubic and cubic graphs with fixed girth are significantly different.
For example, there exists no planar cubic graph with girth at least six, but planar subcubic graphs can have arbitrary large girth.
So Theorem \ref{lz} does not imply Theorem \ref{subcubic girth 3}.

In this paper, we will prove the following lower bounds for the size of largest induced forests in subcubic graphs of girth at least four or five, respectively, in terms of the number of vertices and edges.

\begin{thm} \label{weak main}
Let $G$ be a connected subcubic graph on $n$ vertices and $m$ edges.
	\begin{enumerate}
		\item There exists a finite family ${\mathcal S}$ of subcubic but non-cubic graphs such that if $G$ is triangle-free and $G \not \in \{Q_3,V_8\}$, then $a(G) \geq n-\frac{2}{9}m$, unless $a(G)=n-\frac{2}{9}m-\frac{1}{9}t$ for some $t \in \{1,2\}$ and $G$ is obtained from a disjoint union of $2-t$ copies of graphs in ${\mathcal S}$ and copies of $L$ by adding edges such that each added edge is a cut-edge.
		\item There exist two graphs $R_1,R_2$ and a finite family ${\mathcal S'}$ of subcubic graphs such that if $G$ has girth at least five and $G \not \in \{R_1,R_2\}$, then $a(G) \geq n-\frac{1}{5}m$, unless $a(G)=n-\frac{1}{5}m-\frac{1}{5}$ and $G$ is obtained from a disjoint union of copies of graphs in ${\mathcal S'}$ by adding edges such that each added edge is a cut-edge.
	\end{enumerate}
\end{thm}

The graphs $R_1,R_2$ and families ${\mathcal S},{\mathcal S'}$ will be explicitly described in later sections. 
See Theorem \ref{explicit weak main}.

Theorem \ref{weak main} is tight in the sense that those bounds are attained by infinitely many 2-connected graphs.
Every 2-connected graph obtained from a disjoint union of copies of $L$ (or $R$, respectively) by adding edges between different copies has girth four (or five, respectively) and has no induced forest on more than $n-\frac{2}{9}m$ vertices (or $n-\frac{1}{5}m$ vertices, respectively), where in this paper $R$ denotes the graph obtained from the Petersen graph by deleting an edge.

We remark that as pointed out by Michael Gentner (private communication) after a version of this paper was submitted, the bound $n-\frac{2}{9}m$ for the triangle-free case of Theorem \ref{weak main} can be obtained by a result of Gentner and Rautenbach \cite{gr} when $G$ does not contain $L$ as a subgraph.

Since $m=\frac{3}{2}n_3+n_2+\frac{1}{2}n_1$, where $n_i$ is the number of vertices of degree $i$ in $G$ for each $i \in \{0,1,2,3\}$, $n-\frac{2}{9}m=\frac{2}{3}n_3+\frac{7}{9}n_2+\frac{8}{9}n_1+n_0$ and $n-\frac{1}{5}m=\frac{7}{10}n_3+\frac{4}{5}n_2+\frac{9}{10}n_1+n_0$.
So an immediate corollary of Theorem \ref{weak main} generalizes the case $g=4$ of Theorem \ref{lz} to subcubic graphs and improves the case $g=5$ of Theorem \ref{lz}.

\begin{cor} \label{weak cor}
Let $G$ be a connected subcubic graph on $n$ vertices.
	\begin{enumerate}
		\item If $G$ is triangle-free and $G \not \in \{Q_3,V_8\}$, then $a(G) \geq \frac{2}{3}n$. 
		\item There exist two graphs $R_1,R_2$ and a finite family ${\mathcal S'}$ of subcubic graphs such that if $G$ has girth at least five and $G \not \in \{R_1,R_2\}$, then $a(G) \geq \frac{7}{10}n$, unless $a(G) = \frac{7}{10}n-\frac{1}{5}$ and $G$ is obtained from a disjoint union of copies of graphs in ${\mathcal S'}$ by adding edges such that each added edge is a cut-edge.
	\end{enumerate}
\end{cor}

In fact, we will prove Theorem \ref{real main thm}, which is a stronger version of Theorem \ref{weak main} (and hence of Corollary \ref{weak cor}), in which we allow cycles with length less than four or five as long as they are not disjoint.
However, there are more exceptional graphs than in Theorem \ref{weak main} that should be considered if we allow cycles of length less than four or five.
In order to obtain a clean and simple statement, we will introduce an error function that takes non-zero values only when the graphs are exceptional.
In addition, introducing the error function also provides us a way to give a unified proof for the case of girth four and five.
But the formal definition of this error function is involved, so we postpone its formal definition and the statement of Theorem \ref{real main thm} until Section \ref{sec:error terms}.

The paper is organized as follows.
In Section \ref{sec:basic operation}, we will describe two basic operations that will be used for constructing the families ${\mathcal S},{\mathcal S'}$ mentioned in Theorem \ref{weak main} and show that the graphs that can be constructed by these two operations have a large induced forest with certain strong properties.
These strong properties are important in the proof of Theorem \ref{real main thm}.
In Section \ref{sec:special families}, we will investigate the structure of the graphs that are constructed by the two basic operations mentioned earlier.
We will state Theorem \ref{real main thm} and define the error function in Section \ref{sec:error terms}.
We will also prove that Theorem \ref{real main thm} holds for graphs constructed by the two basic operations in a stronger sense in Section \ref{sec:error terms}, which will be used in Section \ref{proof section} and whose proof provides the intuition of the definition of the error functions.
In Section \ref{sec:precise value}, we will compute the values of the error function and show how to derive Theorem \ref{weak main} from Theorem \ref{real main thm}.
The understanding of explicit values of the error function is another crucial step for proving Theorem \ref{real main thm}.
In Section \ref{sec:dodecahedron}, we will provide a sufficient and necessary condition for a cubic graph being the dodecahedron. 
It is easy to see that the dodecahedron satisfies Theorem \ref{real main thm}.
The strategy for proving Theorem \ref{real main thm} is to obtain a contradiction by showing that every minimum counterexample of Theorem \ref{real main thm} must satisfy the sufficient and necessary condition and hence is the dodecahedron.
This characterization for the dodecahedron might be of independent interest.
Finally, we prove Theorem \ref{real main thm} in Section \ref{proof section}.

We remark that since $a(G)+\phi(G)=\lvert V(G) \rvert$ for all graphs $G$, the bound mentioned in Theorems \ref{weak main} and \ref{real main thm} can be more simply stated if we consider feedback vertex sets.
So in the rest of the paper, we will investigate the minimum size of feedback vertex sets instead of the largest size of induced forests.

To complete this section, we define some notations and notions that will be frequently used in this paper.
If $X$ is a subset of $V(G)$ of a graph $G$, then $G[X]$ denotes the subgraph of $G$ induced by $X$, and $G-X$ denotes the subgraph induced by $V(G)-X$.
When $X$ consists of one vertex, say $v$, then we also write $G-X$ as $G-v$.
If $x,y$ is a pair of non-adjacent vertices of $G$, then $G+xy$ is the graph obtained from $G$ by adding the edge $xy$.
A vertex $v$ in a graph $G$ is a {\em cut-vertex} if $G- v$ has more connected components than $G$.  
Similarly, an edge $e$ in a graph $G$ is a {\em cut-edge} if $G- e$ has more connected components than $G$.
A {\em block} of a graph is a maximal connected subgraph without a cut-vertex.  We say that a block is {\em nontrivial} if it is not isomorphic to $K_1$ or $K_2$.  
Note that in any subcubic graph, every pair of nontrivial blocks is disjoint.
An {\em end-block} of a graph $G$ is a block containing at most one cut-vertex of $G$.  
Observe that every graph that is not 2-connected contains at least two end-blocks.

\section{Basic operations} \label{sec:basic operation}
The special families of graphs mentioned in Theorem \ref{weak main} will be constructed by two basic operations.
This section addresses the relationship between the size of feedback vertex sets in the graph and in the graph obtained by these operations.
These results will be used in follow-up sections.

The first operation is subdividing an edge once.

\begin{lemma}\label{lemma:subdivideedgestrong}
Let $H$ be a graph, $k\in \{0,1\}$, and let $G$ be the graph obtained from $H$ by subdividing an edge.  
If there exists $p\in\mathbb R$ such that $\phi(H- W)\leq p$ for every $W\subseteq E(H)$ with $|W| = k$, then $\phi(G- W')\leq p$ for every $W'\subseteq E(G)$ with $|W'| = k$.
\end{lemma}

\begin{proof}
Suppose $G$ is obtained from $H$ by subdividing an edge $e$.  Say $v$ is the new vertex, and $e_1$ and $e_2$ are the new edges.  
If $W'\subseteq E(H)-\{e\}$, then every minimum feedback vertex set of $H- W'$ is a feedback vertex set of $G-W'$ with size $\phi(H-W')$.
If $W' \subseteq \{e_1,e_2\}$, then every minimum feedback vertex of $H- e$ is a feedback vertex set of $G-W'$ with size $\phi(H-e)$.
This proves the lemma.
\end{proof}

Now we define the second operation.
Let $G$ be a graph.
For $e_1,e_2 \in E(G)$ (not necessarily distinct) and $a \in V(G)$ with degree two, we define $G \circ (e_1,e_2,a)$ to be the graph obtained from $G$ by subdividing $e_1$ and $e_2$ once, respectively, and adding a new vertex adjacent to $a$ and the two vertices obtained from subdividing $e_1$ and $e_2$.
Note that when $e_1=e_2$, it is subdivided twice.

\begin{lemma}\label{lemma:op1edgestrong}
Let $H$ be a graph, $k\in\{0,1\}$, $a\in V(H), e_1,e_2\in E(H)$, and $G = H\circ (e_1,e_2,a)$.  
If there exists $p\in\mathbb R$ such that $\phi(H- W) \leq p$ for every $W\subseteq E(H)$ with $|W| = k$, then $\phi(G- W')\leq p + 1$ for every $W'\subseteq E(G)$ with $|W'| = k$.
\end{lemma}

\begin{proof}
Let $v_1,v_2$ be the new vertices obtained by subdividing edges $e_1,e_2$, respectively, and let $v$ be the vertex in $V(G) - (V(H) \cup \{v_1,v_2\})$.
If $W'\subseteq E(H)$, then adding $v$ into any minimum feedback vertex set of $H -  W'$ is a feedback vertex set of $G-W'$, as desired.
So to prove this lemma, it suffices to prove the case when $W'\not\subseteq E(H)$.  
In particular, $|W'| = 1$.

Let $e$ be the edge in $W'$.  
Suppose $e$ is incident with $v$.  
By symmetry, we may assume without loss of generality that $e$ is not incident to $v_2$.
Then adding $v_2$ into any minimum feedback vertex set of $H-e_2$ is a feedback vertex set of $G-W'$ of size $\phi(H-e_2)+1$, as desired.  
Therefore we may assume $e$ is not incident with $v$.  
By symmetry, we may assume without loss of generality that $e$ is incident with $v_2$ and not $v_1$.
Then adding $v$ into any minimum feedback vertex set of $H-e_2$ is a feedback vertex set of $G-W'$ of size $\phi(H-e_2)+1$, as desired.
\end{proof}

\begin{lemma}\label{lemma:edgestrongimpliesvertexstrong}
Let $G$ be a graph.  If $e\in E(G)$, then for every $v\in V(G)$ incident with $e$, $\phi(G-v)\leq \phi(G-e)$.
\end{lemma}

\begin{proof}
Let $v\in V(G)$ be incident with $e$.  Let $S$ be a feedback vertex set of $G-e$ of size $\phi(G-e)$.  Then $S$ is a feedback vertex set of $G-v$, so $\phi(G-v)\leq \phi(G- e)$, as desired.
\end{proof}

\section{Special families of graphs} \label{sec:special families}
In this section we introduce the special families of graphs in order to formally state Theorem \ref{real main thm}. 

Let $C_1$ be the multigraph on one vertex with one loop.
Define $\F_{1,0}=\{C_1\}$ and define $\F_{i,j}=\emptyset$ if $i\leq0, j<0,$ or $j > i$.
For $0\leq j \leq i$, we define $\F_{i,j}$ to be the set of subcubic multigraphs that can be either obtained from a multigraph in $\F_{i-1,j}$ by subdividing an edge once or obtained from a graph in $\F_{i,j-1}$ by taking operation $\circ$.

Note that every graph in $\F_{i,j}$ has minimum degree at least two for every $i,j$.
Since the operation $\circ$ decreases the number of vertices of degree two, for every $i,j$, $\F_{i,j}$ has exactly $i-j$ vertices of degree two.
Hence a graph in $\F_{i,j}$ is cubic if and only if $i = j$.

{\em Suppressing} a vertex of degree two of a graph means to contract an edge incident with it.
The following lemma simplifies the process of generating the family $\F_{i,j}$.

\begin{lemma} \label{commuting operations}
Let $i,j$ be integers with $i \geq 1$ and $0 \leq j \leq i$.
If $H \in \F_{i,j}$ has at least two vertices and $v$ is a vertex of $H$ of degree two, then $i+j \geq 2$, $i-1 \geq j$, and the multigraph obtained from $H$ by suppressing $v$ belongs to $\F_{i-1,j}$.

In particular, if $i-1 \geq j$, then every multigraph in $\F_{i,j}$ can be obtained from a multigraph in $\F_{i-1,j}$ by subdividing an edge.
\end{lemma}

\begin{proof}
Note that $i>j$, or else $H$ is cubic, and $i+j \geq 2$, or else $H$ is the loop.
We shall prove this lemma by induction on $i+j$.
When $j=0$, it is clear that $H$ is the cycle of length $i$.
Since $H$ has at least two vertices, $i \geq 2$, so suppressing $v$ leads to a cycle of length $i-1$, which belongs to $\F_{i-1,j}$.
This proves the base case.

Now assume $j \geq 1$.
We first assume that $H$ is obtained from a multigraph $H_1 \in \F_{i-1,j}$ by subdividing an edge $e$.
By the induction hypothesis, suppressing $v$ in $H_1$ leads to a multigraph $H_2$ in $\F_{i-2,j}$.
Then the graph obtained from $H$ by suppressing $v$ is the graph obtained from $H_2$ by subdividing the edge $e$ and hence belongs to $\F_{i-1,j}$, as desired.

So we may assume that $H$ is obtained from a multigraph $H_3 \in \F_{i,j-1}$ by taking the operation $\circ$.
By the induction hypothesis, suppressing $v$ in $H_3$ leads to a multigraph $H_4$ in $\F_{i-1,j-1}$.
Then the graph obtained from $H$ by suppressing $v$ is the graph obtained from $H_4$ by taking the operation $\circ$.
\end{proof}

Bearing Lemma \ref{commuting operations} in mind, the following lemma can be easily verified.

\begin{lemma} \label{basic F}
Let $i,j$ be integers with $i \geq 1$ and $0\leq j \leq i$.
The following holds.
	\begin{enumerate}
		\item $\F_{i,0}$ consists of the cycle of length $i$.
		\item All members of $\F_{i,j}$ are simple, unless $i \in \{1, 2\}$ and $j=0$.
		\item $\F_{1,1}=\{K_4\}$.
		\item $\F_{2,1}=\{K_4^+\}$.
		\item $\lvert \F_{3,1} \rvert = 3$; every graph in $\F_{3,1}$ can be obtained from $K_4^+$ by subdividing an edge; $L$ is the only graph in $\F_{3,1}$ with girth at least four.  If $H\in\F_{3,1}$ and $e\in E(H)$, then $H-e$ has girth at most four.
		\item\label{F_{2,2}} Every graph in $\F_{2,2}$ contains disjoint cycles of length less than five; every graph in $\F_{2,2}-\{Q_3,V_8\}$ contains a triangle disjoint from a cycle of length at most four.		
		\item Every graph in $\F_{3,2} \cup \F_{4,1}$ has girth at most four.
		\item If $G\in\F_{i,j}$, then $|V(G)| = i+3j$ and $|E(G)| = i + 5j$.
	\end{enumerate}
\end{lemma}

Recall that $R$ is the graph obtained from the Petersen graph by deleting an edge.
Note that $R$ has girth five and can be obtained from $V_8$ by subdividing the edges in a matching of size 2, so $R \in \F_{4,2}$.

\begin{lemma}\label{girth 5 i=4}
If $G$ is a member of $\F_{4,j}$ with girth at least five, then $j\geq 2$.
Furthermore, $R$ is the only member of $\F_{4,2}$ with girth at least five.
\end{lemma}

\begin{proof}
By Lemma \ref{basic F}, $j \geq 2$.
Now we assume $j=2$.  
By Lemma \ref{commuting operations}, $G$ can be obtained from a graph $H_2\in\F_{2,2}$ by subdividing two edges, not necessarily distinct.  
By Lemma \ref{basic F}, $H_2$ is isomorphic to $Q_3$ or $V_8$, or $H_2$ contains a triangle disjoint from a cycle of length at most four.  Since $G$ has girth at least five, $H_2$ is isomorphic to $V_8$.  
It is not hard to see that $R$ is the unique graph with girth at least five that can be obtained from $V_8$ by subdividing two edges.
\end{proof}

We define $R_1$ and $R_2$ to be the graphs depicted in Figure \ref{girth 5 F_{3,3}}.
They can be obtained from $Q_3$ and $V_8$ respectively by subdividing an edge and applying operation $\circ$, so $R_1,R_2 \in \F_{3,3}$.

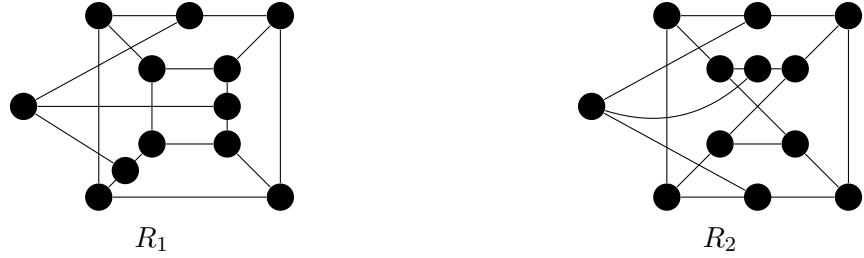
\begin{figure}
\begin{minipage}[b]{.5\linewidth}
	\centering
	\begin{tikzpicture}
		\tikzstyle{every node}=[fill=black, circle]
		\node (a) at (0, 0) {}; \node (1) at ($(a) + (135:1)$) {};
		\node (b) at (1, 0) {}; \node (2) at ($(b) + (45:1)$) {};
		\node (c) at (1, -1) {}; \node (3) at ($(c) + (-45:1)$) {};
		\node (d) at (0, -1) {}; \node (4) at ($(d) + (-135:1)$) {};
		
		\draw (1) -- (2) -- (3) -- (4) -- (1);
		\draw (a) -- (b) -- (c) -- (d) -- (a);
		
		\draw (1) -- (a);
		\draw (2) -- (b);
		\draw (3) -- (c);
		\draw (4) -- (d);
		
		\node (v1) at ($(4)!.5!(d)$) {};
		\node (v2) at ($(b)!.5!(c)$) {};
		\node (v3) at ($(1)!.5!(2)$) {};
		\node (v) at ($(1)!.5!(4) + (-1,0)$) {};
		
		\draw (v) -- (v1); \draw (v) -- (v2); \draw (v) -- (v3);
	\end{tikzpicture}\\$R_1$
\end{minipage}%
\begin{minipage}[b]{.5\linewidth}
	\centering
	\begin{tikzpicture}
		\tikzstyle{every node}=[fill=black, circle]
		\node (a) at (0, 0) {}; \node (1) at ($(a) + (135:1)$) {};
		\node (b) at (1, 0) {}; \node (2) at ($(b) + (45:1)$) {};
		\node (c) at (1, -1) {}; \node (3) at ($(c) + (-45:1)$) {};
		\node (d) at (0, -1) {}; \node (4) at ($(d) + (-135:1)$) {};
		
		\draw (1) -- (2) -- (3) -- (4) -- (1);
		\draw (a) -- (b) -- (d) -- (c) -- (a);
		
		\draw (1) -- (a);
		\draw (2) -- (b);
		\draw (3) -- (c);
		\draw (4) -- (d);
		
		\node (v1) at ($(4)!.5!(3)$) {};
		\node (v2) at ($(a)!.5!(b)$) {};
		\node (v3) at ($(1)!.5!(2)$) {};
		\node (v) at ($(1)!.5!(4) + (-1,0)$) {};
		
		\draw (v) -- (v1); \draw (v) to[bend right] (v2); \draw (v) -- (v3);
	\end{tikzpicture}\\$R_2$
\end{minipage}
\caption{The girth five graphs in $\F_{3,3}$}
\label{girth 5 F_{3,3}}
\end{figure}

\begin{lemma}\label{girth 5 i=3}
If $G\in\F_{3,j}$ has girth at least five, then $j=3$ and $G$ is isomorphic to $R_1$ or $R_2$.
\end{lemma}

\begin{proof}
By Lemma \ref{basic F}, $j=3$,
Therefore $G$ is a cubic graph on twelve vertices.  
It is known that there are only two cubic graphs on twelve vertices with girth at least five.  
(See \cite{bccs}, page 8.)
So there are at most two graphs in $\F_{3,3}$ having girth at least five.
It is easy to see that $R_1$ and $R_2$ are non-isomorphic graphs with girth five.
Hence they are the graphs in $\F_{3,3}$ with girth at least five.
\end{proof}

For integers $i,j,k$ with $i \geq 1$, $0 \leq j \leq i$ and $k \geq 0$, define $\F^4_{i,j,k}$ to be the set of 2-connected subcubic multigraphs obtained from $k$ disjoint copies of $L$ and a copy of a graph in $\F_{i,j}$ by adding $k+1$ edges; define $\F^5_{i,j,k}$ to be the set of 2-connected subcubic multigraphs obtained from $k$ disjoint copies of $R$ and a copy of a graph in $\F_{i,j}$ by adding $k+1$ edges. 
Note that when $k \ge 1$, each added edge must have ends in different copies; when $k=0$, every member of $F_{i,j,k}^4 \cup \F_{i,j,k}^5$ is obtained from a graph in $\F_{i,j}$ by adding an edge.
Observe that for $g\in\{4, 5\}$, $\F_{i,j,k}^g$ is nonempty only when $i-j \geq 2$ since each of its members is subcubic.
Since $K_{3,3}$ can be obtained from $L$ by adding an edge and $L \in \F_{3,1}$, we know $K_{3,3}\in \F^4_{3,1,0} \cap \F^5_{3,1,0}$.  

\section{The error functions} \label{sec:error terms}

Define $t_4 = \frac{2}{9}$ and $t_5 = \frac{1}{5}$.
We need the following Lemmas to ensure the error functions needed to formally state Theorem \ref{real main thm} are well-defined.

\begin{lemma} \label{well-defined 1}
Let $g\in\{4,5\}$.
If $G \in \F^g_{3,j,k}\cap \F_{i',j'}$ for some nonnegative integers $j,k,i',$ and $j'$ with $j \leq 3$, $1 \leq i' \leq 4$ and $j' \leq i'$, then
\begin{equation*}
  1 - t_g + j(1 - 5t_g) - it_g = 1+j'(1-5t_g)-i't_g,
\end{equation*}
where $i=3$.
\end{lemma}

\begin{proof}
Note that $0 \leq j \leq 1$.  We first show $j=1$.  Suppose to the contrary that $j = 0$.
Since $G \in \F^g_{3,0,k}$, $G$ contains a triangle that contains a vertex $v$ of degree two in $G$.
So $i'-j' \geq 1$.
By Lemma \ref{commuting operations}, the mulitgraph $G'$ obtained from $G$ by suppressing $v$ belongs to $\F_{i'-1,j'}$.
But $G'$ is not simple, so $(i',j')=(2,0)$ or $(3,0)$.
However, every member of $\F^g_{3,0,k}$ has more than three edges, a contradiction.

Therefore $j=1$, as claimed.  Hence $G$ is cubic, and this implies that $i'=j'$.

Suppose that $g=4$.
Since $G \in \F_{3,1,k}^4$, $\lvert V(G) \rvert = 6k+6$ and $\lvert E(G) \rvert = 8k+8+(k+1) = 9k+9$.
But $G \in \F_{i',i'}$, so $6k+6 = 4i'$ and $9k+9 = 6i'$.
Hence, $3k+3=2i'$.
But $i' \leq 4$, so $k=1$ and $i'=3$.
Then $1 - t_g + j(1 - 5t_g) - it_g = 1+j'(1-5t_g)-i't_g$.

So $g=5$.
Since $G \in \F_{3,1,k}^5$, $\lvert V(G) \rvert = 10k+6$ and $\lvert E(G) \rvert = 14k+8+(k+1) = 15k+9$.
But $G \in \F_{i',i'}$, so $10k+6=4i'$ and $15k+9=6i'$.
That is, $5k+3=2i'$.
Since $1 \leq i' \leq 4$, $k=1$ and $i'=4$.
Then $1 - t_g + j(1 - 5t_g) - it_g = 1+j'(1-5t_g)-i't_g$.
\end{proof}

\begin{lemma} \label{well-defined 2}
Let $g\in\{4,5\}$.
If $G \in \F^g_{i,j,k}\cap \F_{i',j'}$ for some nonnegative integers $i,j,k,i',$ and $j'$ with $j \leq i$, $1 \leq i' \leq 4$ and $j' \leq i'$, then 
$$\max\{1 - t_g + j(1 - 5t_g) - it_g,0\} = \max\{1+j'(1-5t_g)-i't_g,0\}.$$
\end{lemma}

\begin{proof}
Let $x=\max\{1 - t_g + j(1 - 5t_g) - it_g,0\}$ and $y=\max\{1+j'(1-5t_g)-i't_g,0\}$.
If $k=0$, then by comparing the number of vertices and edges of the graphs in $\F_{i,j,0}^g$ and $\F_{i',j'}$, we obtain that $i'+3j'=i+3j$ and $i'+5j'=i+5j+1$, which implies that $2j'=2j+1$, contradicting that $j$ is an integer.
So $k \geq 1$.
Hence, $|E(G)| \geq 10$.

Recall that $\F_{i,j,k}^g \neq \emptyset$ only when $i-j \geq 2$.
If $i=2$, then $j=0$ and every graph in $\F_{i,j,k}^g$ has a pair of parallel edges, so $(i',j')=(2,0)$, but then $i'+5j'=2<10$.
So $i \geq 3$.
Furthermore, if $i=3$, then $x=y$ by Lemma \ref{well-defined 1}.
Therefore, $i \geq 4$.

We first assume $g=4$.
Then $x = \max\{\frac{7-2i-j}{9}, 0\}$, and $y = \max\{\frac{9-2i'-j'}{9}, 0\}$.
Since $k \geq 1$ and $G \in \F^g_{i,j,k}$, $\lvert E(G) \rvert \geq i+5j+10$; since $G \in \F_{i',j'}$, $\lvert E(G) \rvert = i'+5j'$.
As $i \geq 4$, we know $i'+5j' \geq 14+5j \geq 14$.
So either $i' \geq 5$, or $i'=4$ and $j' \geq 2$, or $i'=3$ and $j' \geq 3$.
In any case, $y=0$ and $x=0$.

So we may assume $g=5$.
Then $x = \max\{\frac{4-i}{5}, 0\}$, and $y = \max\{\frac{5-i'}{5}, 0\}$.
Since $k \geq 1$, $\lvert E(G) \rvert \geq i+5j+16=5j+20$.
So $i'+5j' \geq 5j+20$.
Hence either $i' \geq 5$, or $i'=j'=4$ and $j=0$.
But when $i \geq 4, j=0$, every graph in $\F^g_{i,j,k}$ is not cubic, but every graph in $\F_{4,4}$ is.
So $i' \geq 5$ and hence $x=y=0$.
\end{proof}

We are now ready to define the error functions necessary to formally state Theorem \ref{real main thm}.  
For $g\in\{4,5\}$ and a multigraph $G$, we define $r_g(G)$ to be the sum of $\epsilon_g(B)$ over the blocks $B$ of $G$, where $\epsilon_g(B)$ is defined as follows.
$$\epsilon_g(B) = \left\{
\begin{array}{l l}
\max\{1- j(5t_g-1) - it_g,0\} & \mbox{if } B\in \F_{i,j}, i\geq 1, i \geq j\geq 0,\\
\max\{1 - t_g - j(5t_g-1) - it_g, 0\} & \mbox{if } B\in\F^g_{i,j,k}, i\geq 1, i \geq j \geq 0, k\geq 0,\\
-t_g & \mbox{if } B = K_2,\\
0 & \mbox{otherwise.}
\end{array}\right.$$
Note that $\epsilon_g(B)$ is well-defined by Lemma \ref{well-defined 2}.

Theorem \ref{real main thm} is the main theorem in this paper.
It will be proved in Section \ref{proof section}.

\begin{thm} \label{real main thm}
Let $g \in \{4,5\}$.
If $G$ is a subcubic graph with no two disjoint cycles with length less than $g$, then $\phi(G) \leq t_g\lvert E(G) \rvert + r_g(G)$.
\end{thm}

To close this section, we prove that Theorem \ref{real main thm} holds for graphs in $\F_{i,j}$ for all nonnegative integers $i,j$ in a stronger sense, which will be used in Section \ref{proof section}.
The proof also provides the intuition of the definition of the functions $\epsilon_g$.

\begin{lemma} \label{edge property F}
Let $g\in\{4,5\}$.  
If $G \in \F_{i,j}$ for some nonnegative integers $i,j$ with $j \leq i$, then $\phi(G-e) \leq t_g|E(G)| + \epsilon_g(G) - 1$ for every edge $e \in E(G)$.
\end{lemma}

\begin{proof}
We shall prove this lemma by induction on $i+j$.
When $j=0$, $G$ is a cycle of length $i$, so $\phi(G-e)=0 \leq t_g \lvert E(G) \rvert - 1 + \epsilon_g(G)$.
This prove the base case.  
Therefore we may assume that the lemma holds for all graphs in $\F_{i',j'}$ with $i'+j'<i+j$ or $j'=0$.

We first assume that there exists a multigraph $H$ in $\F_{i-1,j}$ such that $G$ is obtained from $H$ by subdividing an edge.
By the induction hypothesis, $\phi(H-f) \leq t_g|E(H)| + \epsilon_g(H) - 1$  for every $f \in E(H)$.
By Lemma \ref{lemma:subdivideedgestrong}, for every $e \in E(G)$, 
$$\phi(G-e) \leq t_g|E(H)| + \epsilon_g(H) - 1 \leq t_g(|E(G)| - 1) + (\epsilon_g(G) + t_g) - 1= t_g|E(G)| + \epsilon_g(G) - 1,$$
as desired.

So we may assume that there exists a multigraph $H'$ in $\F_{i,j-1}$ such that $G$ is obtained from $H'$ by taking the operation $\circ$.
By the induction hypothesis, $\phi(H'-f) \leq t_g|E(H')| + \epsilon_g(H') - 1$ for every $f \in E(H')$.
By Lemma \ref{lemma:op1edgestrong}, for every $e \in E(G)$, 
$$\phi(G-e) \leq t_g|E(H')| + \epsilon_g(H') \leq t_g(|E(G)| - 5) + (\epsilon_g(G) + 5t_g - 1) = t_g|E(G)| + \epsilon_g(G) - 1,$$
as desired.
\end{proof}

\begin{lemma} \label{vertex property F}
Let $g\in\{4,5\}$, and let $i,j,k$ be nonnegative integers with $j \leq i$ and $k \geq 1$.  
If $G\in\F_{i,j} \cup \F_{i,j,k}^g$, then for all $v\in V(G)$, $G$ admits a feedback vertex set $S$ of size at most $t_g|E(G)| + \epsilon_g(G)$ that contains $v$.
\end{lemma}

\begin{proof}
We first assume that $G\in\F_{i,j}$.  
By Lemmas \ref{lemma:edgestrongimpliesvertexstrong} and \ref{edge property F}, $\phi(G-v) \leq t_g|E(G)| + \epsilon_g(G) - 1$.  
So $G-v$ admits a feedback vertex set $S$ of size at most $t_g|E(G)| + \epsilon_g(G) - 1$.  
Hence, $S\cup\{v\}$ is a feedback vertex set of $G$ of size at most $t_g|E(G)| + \epsilon_g(G)$.

Hence we may assume that $G \in \F_{i,j,k}^g$.
So there exist a multigraph $H_0\in\F_{i,j}$ and graphs $H_1,\dots, H_k$ each isomorphic to $L$ if $g=4$ and isomorphic to $R$ if $g=5$ such that $G$ is obtained from the disjoint union of $H_0,\dots, H_k$ by adding $k+1$ edges.
Since $k\geq 1$, there is some index $p$ such that $v\notin V(H_p)$.  
Choose $u\in V(H_p)$ such that $u$ is adjacent to a vertex in $V(G)-V(H_p)$. 
Recall that $L$ and $R$ are members of $\bigcup_{i',j'}\F_{i',j'}$.
It is easy to see that $\epsilon_4(L)=t_4$ and $\epsilon_5(R)=t_5$.
So for each $\ell$ with $0 \leq \ell \leq k$, $H_\ell$ admits a feedback vertex set $S_\ell$ of size at most $t_g\lvert E(H_\ell) \rvert + \epsilon_g(H_\ell)$ such that $u\in S_p$ and $v\in S_{\ell'}$ for some $\ell'$.  
Then $\bigcup_{\ell=0}^kS_\ell$ is a feedback vertex set of $G$ with size at most 
\begin{multline*}
  \sum_{\ell=0}^k (t_g\lvert E(H_\ell) \rvert+\epsilon_g(H_\ell)) = t_g(\lvert E(G) \rvert-k-1)+\epsilon_g(H_0)+kt_g\\
  =t_g\lvert E(G) \rvert+\max\{1+j(1-5t_g)-it_g,0\}-t_g \\
  \leq t_g \lvert E(G) \rvert + \max\{1-t_g+j(1-5t_g)-it_g,0\}
  =t_g\lvert E(G) \rvert + \epsilon_g(G).
\end{multline*}
This completes the proof.
\end{proof}

\section{Precise values for the error terms} \label{sec:precise value}

By the definition,
$$\epsilon_4(B) = \left\{
\begin{array}{l l}
\frac{9-2i-j}{9} & \mbox{if } B\in \F_{i,j} \mbox{ where } 4 \geq i\geq1, i \geq j\geq0, 2i + j \leq 9,\\
\frac{7-2i-j}{9}  & \mbox{if } B\in\F^4_{i,j,k} \mbox{ where } 3 \geq i\geq1, i \geq j \geq 0, k\geq0,\\
-\frac{2}{9} & \mbox{if } B = K_2,\\
0 & \mbox{otherwise.}
\end{array}\right.$$
$$\epsilon_5(B) = \left\{
\begin{array}{l l}
\frac{5-i}{5} & \mbox{if } B\in \F_{i,j} \mbox{ where } 4 \geq i\geq1, i \geq j\geq0,\\
\frac{4-i}{5}  & \mbox{if } B\in\F^5_{i,j,k} \mbox{ where } 3 \geq i\geq1, i \geq j \geq 0, k\geq0,\\
-\frac{1}{5} & \mbox{if } B = K_2,\\
0 & \mbox{otherwise.}
\end{array}\right.$$

The following lemmas describe the values of the functions $\rt,\rf$ for graphs with no two disjoint short cycles.

\begin{lemma} \label{r for girth 4}
If $G$ is a connected subcubic graph with no two disjoint triangles, then one of the following holds.
	\begin{enumerate}
		\item $G \in \{K_3,K_4,K_4^+\} \cup \F_{2,2}$.
		\item $\rt(G)=\frac{4}{9}$, and $G$ is obtained from a disjoint union of a copy of $K_4^+$ and at least one copy of $L$ by adding edges such that each new edge is a cut-edge.
		\item $\rt(G) = \frac{1}{3}$, and either 
			\begin{enumerate}
				\item $G$ is obtained from a disjoint union of a copy of $K_4^+$, a graph in $\F_{3,2} \cup \F_{4,0}$ and copies of $L$ by adding edges such that each new edge is a cut-edge, or 
				\item $G$ is obtained from a disjoint union of $K_3$ and at least one copy of $L$ by adding edges such that each new edge is a cut-edge.
			\end{enumerate}
		\item $\rt(G) = \frac{2}{9}$, and either 
			\begin{enumerate}
				\item some block of $G$ is isomorphic to $K_4^+$ or $K_3$, or 
				\item $G$ is obtained from a disjoint union of copies of graphs in $\F_{3,1}$ by adding edges such that each new edge is a cut-edge.
			\end{enumerate}
		\item $\rt(G) = \frac{1}{9}$, and either 
			\begin{enumerate}
				\item some block of $G$ is isomorphic to $K_4^+$ or $K_3$, or 
				\item $G$ can be obtained from a graph in $\F_{3,2} \cup \F_{4,0} \cup \bigcup_{k \geq 1}\F_{3,0,k}^4$ and copies of members of $\F_{3,1}$ by adding edges such that each new edge is a cut-edge.
			\end{enumerate}
		\item $\rt(G) \leq 0$.
	\end{enumerate}
\end{lemma}

\begin{proof}
We shall prove this lemma by induction on $\lvert V(G) \rvert$.
We are done if $G \in \{K_3,K_4,K_4^+\} \cup \F_{2,2}$.
This proves the induction basis and we may assume that $G$ has at least two vertices and $G \not \in \{K_3,K_4,K_4^+\} \cup \F_{2,2}$.

We first assume that $K_4^+$ is a block of $G$.
Since $G$ does not contain two disjoint triangles, there uniquely exists a block $B$ isomorphic to $K_4^+$, and $G-V(B)$ has no triangle.
Moreover, $B$ is an end-block.
Note that $G \neq B$, so $G-V(B)$ is nonempty and does not contain $K_4^+$ as an end-block.
In addition, $G-V(B) \not \in \{K_3,K_4,K_4^+\} \cup \F_{2,2}$ as it contains a vertex of degree at most two, so $\rt(G-V(B)) \leq \frac{2}{9}$ by the induction hypothesis.
Note that $\rt(G) = \rt(G-V(B)) - \frac{2}{9} + \rt(B)= \rt(G-V(B))+ \frac{2}{9}$.

If $\rt(G-V(B))=\frac{2}{9}$, then $\rt(G)=\frac{4}{9}$, and by the induction hypothesis, $G-V(B)$ is obtained from a disjoint union of copies of graphs in $\F_{3,1}$ by adding edges such that each new edge is a cut-edge.
Since $G-V(B)$ does not contain a triangle, the chosen member in $\F_{3,1}$ must be $L$, by Lemma \ref{basic F}.
So Statement 2 holds.
If $\rt(G-V(B))=\frac{1}{9}$, then $\rt(G)=\frac{1}{3}$, and by the induction hypothesis, since $G-V(B)$ is triangle-free, $G-V(B)$ can be obtained from a graph in $\F_{3,2} \cup \F_{4,0}$ and copies of $L$ by adding edges such that each new edge is a cut-edge.
Then Statement 3 holds.
If $\rt(G-V(B))=0$, then $\rt(G)=\frac{2}{9}$ and $K_4^+$ is an end-block of $G$, so Statement 4 holds.

Hence we may assume that $K_4^+$ is not a block of $G$.
If $G$ has a leaf $v$, then $\rt(G)=\rt(G-v)-\frac{2}{9} \leq \frac{1}{9}$ by the induction hypothesis, since $\rt(K_3) \leq \frac{1}{3}$ and $G-v$ does not contain $K_4^+$ as an end-block.
If $\rt(G)=\frac{1}{9}$, then $\rt(G-v)=\frac{1}{3}$, so $G-v$ (and hence $G$) contains $K_3$ as a block by the induction hypothesis, and Statement 5 holds.
If $\rt(G)<\frac{1}{9}$, then Statement 6 holds.

So we may assume that $G$ has minimum degree at least two.
Since $G \neq K_3$, $G$ has a nontrivial end-block $B'$ that is not isomorphic to $K_3$.
By the induction hypothesis, $\rt(G-V(B')) \leq \frac{1}{3}$ since $G-V(B')$ is not cubic and does not contain $K_4^+$.
Furthermore, since $B'$ is not a triangle, $0 \leq \rt(B') \leq \frac{2}{9}$. 
Hence $\rt(G) = \rt(B') - \frac{2}{9} + \rt(G-V(B'))$. 

Assume that $\rt(G-V(B'))=\frac{1}{3}$.
Then by the induction hypothesis, $G-V(B')$ is obtained from a disjoint union of $K_3$ and copies of $L$ by adding edges such that each added edge is a cut-edge.
In particular, $B'$ is triangle-free since $G$ has no two disjoint triangles.
If $\rt(B')=\frac{2}{9}$, then $B' \in \F_{3,1}$ by induction, but then $B'=L$ by Lemma \ref{basic F}, so $\rt(G) = \frac{1}{3}$ and $G$ is obtained from a disjoint union of $K_3$ and at least one copy of $L$ by adding edges such that each added edge is a cut-edge.
Hence Statement 3 holds when $\rt(B')=\frac{2}{9}$.
If $\rt(B')=\frac{1}{9}$, then $\rt(G)=\frac{2}{9}$ and $G$ contains $K_3$ as a block, so Statement 4 holds.
If $\rt(B')=0$, then $\rt(G) = \frac{1}{9}$ and $G$ contains $K_3$ as a block, so Statement 5 holds.

Therefore, we may assume $\rt(G-V(B')) \leq \frac{2}{9}$. 
We may assume Statement 6 does not hold, so $\frac{1}{3} \leq \rt(G-V(B'))+\rt(B') \leq \frac{4}{9}$.
If $\rt(G-V(B')) = \rt(B') = \frac{2}{9}$, then $\rt(G)=\frac{2}{9}$, and $B' \in \F_{3,1}$, and either $G-V(B')$ (and hence $G$) contains $K_3$ as a block, or $G-V(B')$ (and hence $G$) can be obtained from a disjoint union of copies of graphs in $\F_{3,1}$ by adding edges such that each new edge is a cut-edge.

So we may assume that $\rt(G-V(B'))+\rt(B')=\frac{1}{3}$, and hence $\rt(G)=\frac{1}{9}$.
We may further assume that no block of $G$ is isomorphic to $K_4^+$ or $K_3$, for otherwise Statement 5 holds.
If $\rt(G-V(B'))=\frac{2}{9}$ and $\rt(B')=\frac{1}{9}$, then by the induction hypothesis, $G-V(B')$ can be obtained from a disjoint union of copies of members of $\F_{3,1}$ by adding edges such that each new edge is a cut-edge, and $B' \in \F_{3,2} \cup \F_{4,0} \cup \bigcup_{k \geq 1}\F_{3,0,k}^4$, so Statement 5 holds.
If $\rt(G-V(B'))=\frac{1}{9}$ and $\rt(B')=\frac{2}{9}$, then Statement 5 also holds by the induction hypothesis.
This completes the proof.
\end{proof}

\begin{lemma} \label{r for girth 5}
If $G$ is a connected subcubic graph with no two disjoint cycles of length less than five, then one of the following holds.
	\begin{enumerate}
		\item $G=K_4$ or $G$ contains $K_4^+$ as a subgraph.
		\item $\rf(G) = \frac{2}{5}$ and $G$ is obtained from a disjoint union of a graph in $\bigcup_{j=0}^3 \F_{3,j}$ and copies of graphs in $\bigcup_{j=0}^3 \F_{4,j}$ by adding edges between different graphs such that each new edge is a cut-edge in $G$.
		\item $\rf(G) = \frac{1}{5}$ and either 
			\begin{enumerate}
				\item $G$ contains a block that is isomorphic to a member of $\bigcup_{j=0}^2 \F_{3,j}$, or 
				\item $G \in \F_{3,1,k}^5$ for some nonnegative integer $k$, or 
				\item $G$ is obtained from a disjoint union of copies of members of $\bigcup_{j=0}^4 \F_{4,j} \cup \bigcup_{k \geq 0}\F_{3,0,k}^5$ by adding edges between different graphs such that each new edge is a cut-edge in $G$.
			\end{enumerate}
		\item $\rf(G) \leq 0$.
	\end{enumerate}
\end{lemma}

\begin{proof}
By Lemma \ref{basic F}, no block of $G$ is isomorphic to a graph in $\F_{2,2}$.
Suppose that $G$ is 2-connected.
Since $G \not \in \{K_4,K_4^+\}$, $\rf(G)=\ef(G) \leq \frac{2}{5}$.
And if $\rf(G)=\frac{2}{5}$, then $G \in \F_{3,j}$ for some $j \geq 0$, since $\F_{2,j,k}^5$ does not contain simple graphs for any nonnegative integers $j,k$; if $\rf(G)=\frac{1}{5}$, then $G \in \F_{3,j,k}^5 \cup \F_{4,j}$ for some nonnegative integers $j \leq 1,k$.
So this lemma holds if $G$ is 2-connected.

We shall prove this lemma by induction on $\lvert V(G) \rvert$.
The induction basis was proved by the 2-connected case.
We assume that this lemma holds for every connected subcubic graph on fewer vertices with no two disjoint cycles of length less than five.
We may also assume that $G \neq K_4$ and $G$ does not contain $K_4^+$ as a subgraph.

Assume that $G$ contains a leaf $v$.
Then by the induction hypothesis, $\rf(G) = \rf(G-v) - \frac{1}{5} \leq \frac{1}{5}$, and the equality holds only when $\rf(G-v)=\frac{2}{5}$ and hence $G$ contains a graph in $\bigcup_{j=0}^3 \F_{3,j}$ as a block.
Note that that block is not cubic and hence does not belong to $\F_{3,3}$.
So we may assume that $G$ has minimum degree at least two.

Since every member of $\bigcup_{j= 0}^2 (\F_{2,j} \cup \F_{3,j})$ contains a cycle of length less than five by Lemma \ref{basic F}, $G$ contains a nontrivial end-block $B$ not isomorphic to a graph in $\bigcup_{j=0}^2 (\F_{2,j} \cup \F_{3,j})$.
Since $G$ is not 2-connected, $B$ is not cubic and hence not in $\F_{3,3}$.
So $\rf(B) \leq \frac{1}{5}$ by the induction hypothesis.
Hence $\rf(G) = \rf(B)-\frac{1}{5}+\rf(G-V(B)) \leq \rf(G-V(B))$.
By the induction hypothesis, $\rf(G) \leq \frac{2}{5}$.

Furthermore, if $\rf(G)=\frac{2}{5}$, then $\rf(B)=\frac{1}{5}$ and $\rf(G-V(B))=\frac{2}{5}$.
Since $G-V(B)$ is not cubic, $G-V(B)$ is obtained from a disjoint union of a member in $\bigcup_{j=0}^2 \F_{3,j}$ and copies of members of $\bigcup_{j=0}^3 \F_{4,j}$ by adding edges between different graphs such that each new edge is a cut-edge.
Note that every graph in $\bigcup_{j=0}^2 \F_{3,j}$ has girth at most four, so $B$ has girth at least five.
Since $B$ is 2-connected with $\rf(B)=\frac{1}{5}$ and girth at least five and is not cubic, $B \in \bigcup_{j=2}^3 \F_{4,j}$.
Therefore, $G$ is obtained from a disjoint union of a member in $\bigcup_{j=0}^2 \F_{3,j}$ and copies of members of $\bigcup_{j=0}^3 \F_{4,j}$ by adding edges between different graphs such that each new edge is a cut-edge.

In addition, if $\rf(G) = \frac{1}{5}$, then either $\rf(G-V(B))=\frac{2}{5}$, or $\rf(G-V(B))=\rf(B)=\frac{1}{5}$. 
We are done if $G$ contains a block isomorphic to a member of $\bigcup_{j=0}^2\F_{3,j}$.
Hence $\rf(G-V(B))=\rf(B) = \frac{1}{5}$ by the induction hypothesis.
Note that every graph in $\bigcup_{k \geq 0}\F_{3,1,k}^5$ is cubic.
So by the induction hypothesis, $G-V(B)$ is obtained from a disjoint union of copies of members of $\bigcup_{j=0}^4\F_{4,j} \cup \bigcup_{k \geq 0}\F_{3,0,k}^5$ by adding edges between different graphs such that each new edge is a cut-edge in $G-V(B)$.
Since $B$ is 2-connected but not cubic, $B \in \bigcup_{j=0}^3 \F_{4,j} \cup \bigcup_{k \geq 0} \F_{3,0,k}^5$.
Therefore, $G$ is obtained from a disjoint union of copies of members of $\bigcup_{j=0}^4\F_{4,j} \cup \bigcup_{k \geq 0}\F_{3,0,k}^5$ by adding edges between different graphs such that each new edge is a cut-edge in $G$.
This proves the lemma.
\end{proof}

The following theorem is the explicit form of Theorem \ref{weak main}.
We prove it in this section, by assuming Theorem \ref{real main thm}, which will be proved in Section \ref{proof section}.

\begin{thm} \label{explicit weak main}
Let $G$ be a connected subcubic graph on $n$ vertices and $m$ edges.
	\begin{enumerate}
		\item If $G$ has girth at least four, then one of the following holds.
			\begin{enumerate}
				\item $G \in \{Q_3,V_8\}$ and $\phi(G) = \frac{2}{9}m+\frac{1}{3}$.
				\item $G$ is obtained from a disjoint union of copies of $L$ by adding edges such that each new edge is a cut-edge, and $\phi(G) = \frac{2}{9}m+\frac{2}{9}$.
				\item $G$ is obtained from a disjoint union of a member of $\F_{3,2} \cup \F_{4,0}$ and copies of $L$ by adding edges such that each new edge is a cut-edge, and $\phi(G) = \frac{2}{9}m+\frac{1}{9}$.
				\item $\phi(G) \leq \frac{2}{9}m$.
			\end{enumerate}
		\item If $G$ has girth at least five, then one of the following holds.
			\begin{enumerate}
				\item $G \in \{R_1,R_2\}$ and $\phi(G) = \frac{1}{5}m+\frac{2}{5}$.
				\item $G \in \F_{4,4}$ or can be obtained from a disjoint union of members of $\{R\} \cup \F_{4,3}$ with girth at least five by adding edges such that each new edge is a cut-edge, and $\phi(G) \leq \frac{1}{5}m+\frac{1}{5}$.
				\item $\phi(G) \leq \frac{1}{5}m$.
			\end{enumerate}
	\end{enumerate}
\end{thm}

\begin{proof}
We first assume that $G$ has girth at least four.
It is straight forward to compute $\phi(G)$ if $G \in \{Q_3,V_8\}$ or can be obtained from a disjoint union of at most one graph in $\F_{3,2} \cup \F_{4,0}$ and copies of $L$ by adding edges such that each added edge is a cut-edge.
So we assume that $G$ is not of those cases, and we shall prove that $\phi(G) \leq \frac{2}{9}m$.
Since $G$ has girth at least four, $G$ does not contain $K_3, K_4$ or $K_4^+$ as a subgraph.
By Lemma \ref{basic F}, $G \not \in \F_{2,2}$, otherwise $G \in \{Q_3,V_8\}$.
In addition, if $G$ contains a member of $\F_{3,1}$ as a subgraph, then this member must be $L$ by Lemma \ref{basic F}.
Therefore, by Theorem \ref{real main thm} and Lemma \ref{r for girth 4}, $\phi(G) \leq \frac{2}{9}m$.
This proves the first statement.

Now we assume that $G$ has girth at least five and prove the second statement.
By Lemma \ref{basic F}, $G$ does not contain a subgraph isomorphic to $K_4, K_4^+$ or members of $\bigcup_{j=0}^2 (\F_{3,j} \cup \bigcup_{k \geq 0}\F_{3,j,k}^5) \cup \F_{4,0} \cup \F_{4,1}$.
Since every graph in $\F_{3,3}$ is cubic, by Lemma \ref{r for girth 5}, if $\rf(G) = \frac{2}{5}$, then $G \in \F_{3,3}$.
But Lemma \ref{girth 5 i=3} implies that $G \in \{R_1,R_2\}$.
Hence, by Theorem \ref{real main thm}, if $\phi(G) \geq \frac{1}{5}m+\frac{2}{5}$, then $G \in \{R_1,R_2\}$.
And it is straight forward to see $\phi(G) = \frac{m}{5}+\frac{2}{5}$ if $G \in \{R_1,R_2\}$.
If $\rf(G) = \frac{1}{5}$, then by Lemma \ref{r for girth 5}, $G$ can be obtained from a disjoint union of copies of members of $\bigcup_{j=2}^4 \F_{4,j}$ by adding edges such that each added edge is a cut-edge.
By Lemma \ref{girth 5 i=4}, $R$ is the only graph with girth at least five in $\F_{4,2}$.
Note that every graph in $\F_{4,4}$ is cubic, so if $G$ contains $\F_{4,4}$ as a subgraph, then $G \in \F_{4,4}$.
Hence, together with Theorem \ref{real main thm}, if $\phi(G) =\frac{1}{5}m+\frac{1}{5}$, then $G \in \F_{4,4}$ or is obtained from a disjoint union of members of $\{R\} \cup \F_{4,3}$ by adding edges such that each added edge is a cut-edge.
This proves the second statement.
\end{proof}

\section{A characterization of the dodecahedron} \label{sec:dodecahedron}

We shall prove a characterization of the dodecahedron in this section.
It is of independent interest and will be used in Section \ref{proof section} for the proof of Theorem \ref{real main thm}.

\begin{thm} \label{characterization dodecahedron}
A cubic graph $G$ is a disjoint union of copies of dodecahedra if and only if it has girth five and for every vertex $v$ of $G$ and for every pair of neighbors $a,b$ of $v$, $G-v$ contains two disjoint 5-cycles, where one contains $a$ and the other contains $b$.
\end{thm}

\begin{proof}
It is clear that if $G$ is a disjoint union of copies of dodecahedra, then for every vertex $v$ of $G$ and for every pair of neighbors $a,b$ of $v$, $G-v$ contains two disjoint 5-cycles, where one contains $a$ and the other contains $b$.
It suffices to prove the converse statement.
So we assume that for every vertex $v$ of $G$ and for every pair of neighbors $a,b$ of $v$, $G-v$ contains two disjoint 5-cycles, where one contains $a$ and the other contains $b$.

We first claim that $G$ does not have two 5-cycles sharing two edges.
Since $G$ has girth at least five, no two 5-cycles can share more than two edges, and the shared edges must have a common end.  
Suppose to the contrary that there are two 5-cycles in $G$ sharing sharing two edges, say $av$ and $vb$.  
Let $a_1$ and $a_2$ be the neighbors of $a$ other than $v$, and let $b_1$ and $b_2$ be the neighbors of $b$ other than $v$ such that $a_1$ is adjacent to $b_1$ and $a_2$ is adjacent to $b_2$.  

By assumption, $G-v$ contains two disjoint 5-cycles $D_1$ and $D_2$, where $D_1$ contains $a_1aa_2$ and $D_2$ contains $b_1bb_2$.  
Let $D_1 = aa_1 a'_1 a'_2 a_2 a$ and $D_2 = bb_1b'_1 b'_2 b_2b$.
Let $c$ be the neighbor of $v$ other than $a$ and $b$.  
By assumption, $G- a$ contains disjoint 5-cycles $D'_1$ and $D'_2$ such that $D'_1$ contains $a_1'a_1b_1$ and $D_2'$ contains $cvb$.  
Since $D_1'$ and $D_2'$ are disjoint, $b_2\in V(D'_2)$.  
Since $G$ has girth at least 5, $c$ is not adjacent to $a_2$.  
Therefore $c$ is adjacent to $b'_2$.

Again by assumption, $G- a$ contains disjoint 5-cycles $D''_1$ and $D''_2$ such that $D_1''$ contains $a_2'a_2b_2$ and $D_2''$ contains $cvb$.
Therefore $c,v,b,b_1\in V(D''_2)$.  
Since $G$ has girth at least 5, $c$ is not adjacent to $a_1$, so $c$ is adjacent to $b'_1$.
But then $cb'_1b'_2c$ is a triangle in $G$, a contradiction.  
This proves that $G$ does not contain two 5-cycles sharing two edges.

Now we claim that every edge of $G$ is contained in precisely two 5-cycles.
Let $xy\in E(G)$.  
Let $x_1$ and $x_2$ be the neighbors of $x$ other than $y$.  
By assumption, there is a 5-cycle in $G-x_1$, and hence in $G$, containing $xy$.  
Suppose that $G$ contains three 5-cycles $Z_1,Z_2,Z_3$ containing $xy$.
Without loss of generality, we may assume that $Z_1$ contains $x_1xy$.
Since no two 5-cycles can share two edges, $Z_2$ and $Z_3$ do not contain $x_1x$.
So $Z_2$ and $Z_3$ contain both $xx_2$ and $xy$, a contradiction.  
This proves that every edge of $G$ is contained in precisely two 5-cycles.

Since every edge is contained in precisely two 5-cycles and $G$ has an edge, $G$ contains a 5-cycle, say $C = v_0v_1v_2v_3v_4v_0$.  
For integers $i$ with $0 \leq i \leq 4$, let $v'_i$ be the neighbor $v_i$ not in $C$.  
Since $G$ has girth at least five, $v'_i\neq v'_j$ for $i\neq j$.  
For each $i$ with $0 \leq i \leq 4$, $v_iv_{i+1}$ is contained in precisely two 5-cycles (where indices are computed modulo 5), so $v_iv_{i+1}$ is contained in a 5-cycle $U_i$ such that $E(C)\cap E(U_i) = \{v_iv_{i+1}\}$.  
Therefore $v'_i, v'_{i+1}\in V(U_i)$, and there is a vertex $w_i\notin\{v_i:0\leq i\leq 4\}$ adjacent to both $v'_i$ and $v'_{i+1}$.  
Suppose $w_i = w_j$ for $i\neq j$.  
Since $G$ is cubic, $|i - j| = 1$.  
But then $|E(U_i)\cap E(U_j)| = 2$, a contradiction.
So $w_i$ for $0 \leq i \leq 4$ are pairwise distinct.

Let $w'_i$ be the neighbor of $w_i$ not in $U_i$.  
Clearly $w'_i$ is distinct from $v_j$ and $v'_j$ for all $j$.  
Since $v'_iw_i$ is contained in precisely two 5-cycles, there exists a 5-cycle $U_i'$ such that $E(U_i)\cap E(U'_i) = \{v'_iw_i\}$.  
Therefore $U'_i$ contains $w_{i-1}v_i'w_iw_i'$, so $w'_{i-1}$ is adjacent to $w'_i$.  
So $w'_{i-1}\neq w'_{i+1}$, otherwise $\lvert E(U'_i)\cap E(U'_{i+1}) \rvert \geq 2$, a contradiction.  
Since $G$ is cubic, the component of $G$ containing $C$ has no more vertices, so the component of $G$ containing $C$ is the dodecahedron.

Note that $C$ is a 5-cycle that contains an arbitrary edge of $G$ we picked at the beginning.
Since $G$ is cubic, every component of $G$ has an edge, so every component of $G$ is a dodecahedron.
This completes the proof.
\end{proof}

\section{Proof of Theorem \ref{real main thm}} \label{proof section}

In this section we prove Theorem \ref{real main thm}.
We shall prove Theorem \ref{real main thm} by induction on $\lvert V(G) \rvert$.
It is easy to see that Theorem \ref{real main thm} holds for graphs with at most two vertices.

For $g \in \{4,5\}$, we say that a graph $G$ is {\it $g$-minimal} if $G$ is a subcubic graph that does not contain two disjoint cycles of length less than $g$ and $\phi(G) > t_g|E(G)| + r_g(G)$, but $\phi(H) \leq t_g \lvert E(H) \rvert + r_g(H)$ for every subcubic graph $H$ with fewer vertices than $G$ but not containing two disjoint cycles of length less than $g$.
In other words, $g$-minimal graphs are the minimum counterexamples of Theorem \ref{real main thm}.
Note that every $g$-minimal graph contains at least three vertices and does not belong to $\bigcup_{i \geq 1,i \geq j \geq 0}\F_{i,j}\cup\bigcup_{i\geq 1, j,k\geq0} \F^g_{i,j,k}$ by Lemma \ref{vertex property F}.

In the rest of this section, $g \in \{4,5\}$ and $G$ is a $g$-minimal graph.  
We denote $\lvert V(G) \rvert$ and $\lvert E(G) \rvert$ by $n$ and $m$, respectively.

\begin{lemma} \label{2-conn}
$G$ is 2-connected.
\end{lemma}

\begin{proof}
It is easy to see that every $g$-minimal graph is connected, so $G$ is connected.

Suppose that $G$ has a leaf $v$.
By the minimality of $G$, $G-v$ admits a feedback vertex set $S$ of size at most $t_g(m-1)+ r_g(G-v) =t_gm + r_g(G)$.
But $S$ is a feedback vertex set of $G$, a contradiction.  Therefore $G$ has minimum degree at least two.

If $G$ is not 2-connected, then $G$ contains a nontrivial end-block $B$.
Since $G$ is subcubic, there exists a cut-edge of $G$ incident with $B$.
By the minimality of $G$, $B$ admits a feedback vertex set $S_B$ of size at most
$t_g|E(B)| + \epsilon_g(B)$, and $G-V(B)$ admits a feedback vertex set $S_B'$ of size at most $t_g\left(m-\lvert E(B) \rvert - 1\right) + r_g(G-V(B))$.
But $S_B\cup S'_B$ is a feedback vertex set of $G$ of size at most 
$t_gm + r_g(G-V(B)) + \epsilon_g(B) - t_g = t_gm + r_g(G)$, 
a contradiction.  This proves the lemma.
\end{proof}

Recall that $G$ does not belong to $\bigcup_{i \geq 1,i \geq j \geq 0}\F_{i,j}\cup\bigcup_{i\geq 1, j,k\geq0} \F^g_{i,j,k}$ by Lemma \ref{vertex property F}.
By Lemma \ref{2-conn}, $r_g(G) = \epsilon_g(G) = 0$.

An {\em edge-cut} of $G$ is an ordered partition $[A,B]$ of $V(G)$.
The {\em order} of $[A,B]$ is the number of edges with one end in $A$ and one end in $B$.
$[A,B]$ is a {\em splitter} if it has order at most two with $\lvert A \rvert \geq 2$ such that $G[A]$ has girth at least $g$ and $G[B]$ is 2-connected.

Since $G$ is 2-connected, the order of any splitter $[A,B]$ is two.
Since $A$ and $B$ contain at least two vertices, the edges between $A$ and $B$ do not share ends.
In the rest of this section, if $[A,B]$ is a splitter, then we let $u_A$ and $v_A$ be the ends of the edges between $A$ and $B$ in $A$, and we let $u_B$ and $v_B$ be the ends of the edges between $A$ and $B$ in $B$; let $n_A = \lvert A \rvert$, $n_B=\lvert B \rvert$, $m_A = \lvert E(G[A]) \rvert$ and $m_B = \lvert E(G[B]) \rvert$; let $S_A$ and $S_B$ be minimum feedback vertex sets of $G[A]$ and $G[B]$, respectively.
Observe that $S_A$ and $S_B$ have size at most $t_gm_A+ r_g(G[A])$ and $t_gm_B+r_g(G[B])$, respectively, since $G$ is $g$-minimal.
Note that $G[A]$ and $G[B]$ are connected and each contains at least two vertices of degree at most two, since $G$ is 2-connected.

\begin{lemma}\label{no splitter with B in F}
If $[A,B]$ is a splitter, then no feedback vertex set of $G[B]$ with size at most $t_gm_B+r_g(G[B])$ contains at least one of $u_B$ or $v_B$.
In particular,	$G[B]\notin\bigcup_{i \geq 1,j \geq 0}\F_{i,j}\cup\bigcup_{4\geq i\geq 1,j \geq 0,k\geq1} \F^g_{i,j,k}$, and $G[B]$ cannot be obtained from a member of $\bigcup_{4 \geq i\geq 1, j \geq 0, k \geq 1}\F_{i,j,k}^g$ by subdividing an edge.
\end{lemma}

\begin{proof}
Suppose that $S_B$ contains at least one of $u_B$ or $v_B$.
Hence, $S_A \cup S_B$ is a feedback vertex set with size at most $t_g(m_A+m_B)+r_g(G[A])+r_g(G[B]) = t_gm+r_g(G[A]) + r_g(G[B])-2t_g$.
Since $r_g(G) = 0$, $r_g(G[A]) + r_g(G[B]) -2t_g >0$.

Suppose that $g=4$.
Since $G[A]$ has girth at least four, by Lemma \ref{r for girth 4}, $\rt(G[A]) \leq \frac{2}{9}$, and the equality holds only when $G[A]$ is obtained from a disjoint union of $L$ by adding edges between different copies of $L$ such that each added edge is a cut-edge.
By Lemma \ref{r for girth 4}, since $G[B]$ is 2-connected and does not contain $K_4^+$ as a subgraph, $\rt(G[B]) \leq \frac{1}{3}$, and the equality holds only when $G[B]=K_3$.
Since $\rt(G[A])+\rt(G[B])-2t_g > 0$, we have $\rt(G[A])=\frac{2}{9}$ and $\rt(G[B])=\frac{1}{3}$.
Therefore, $G$ is obtained from a disjoint union of a $K_3$ and at least one copy of $L$ by adding edges such that each added edge is between different graphs.
In other words, $G \in \F_{3,0,k}^4$ for some positive integer $k$, a contradiction.
Hence $g=5$.

Since $G[A]$ has firth at least five, $\rf(G[A]) \leq \frac{1}{5}$ by Lemma \ref{r for girth 5}.
Note that $G[A]$ contains at least two vertices of degree at most two.
Together with Lemma \ref{basic F}, $\rf(G[A])=\frac{1}{5}$ only when $G[A]$ is obtained from a disjoint union of copies of members of $\F_{4,2}$ with girth at lest five by adding edges such that each added edge is a cut-edge.
Recall that $R$ is the only member of $\F_{4,2}$ with girth at least five by Lemma \ref{girth 5 i=4}.
Since $G[B]$ does not contain $K_4^+$ as a subgraph but contains at least two vertices of degree two, $\rf(G[B]) \leq \frac{2}{5}$, and the equality holds only when $G[V(B)] \in \F_{3,0} \cup \F_{3,1}$.
Since $\rt(G[A])+\rt(G[B])-2t_g > 0$, we have $\rt(G[A])=\frac{1}{5}$ and $\rt(G[B])=\frac{2}{5}$.
Therefore, $G \in \F_{3,0,k}^5 \cup \F_{3,1,k}^5$ for some positive integer $k$, a contradiction.
This proves that no feedback vertex set of $G[B]$ with size at most $t_gm_B-r_g(G[B])$ contains at least one of $u_B$ or $v_B$.

In particular, by Lemma \ref{vertex property F}, $G[B] \not \in \bigcup_{i=1}^4\bigcup_{j \geq 0}\F_{i,j}\cup\bigcup_{4\geq i\geq 1,j \geq 0,k\geq1} \F^g_{i,j,k}$.
If $G[B]$ is obtained from a member of $\bigcup_{4\geq i\geq 1,j \geq 0,k\geq1} \F^g_{i,j,k}$ by subdividing an edge, then at least one of $u_B,v_B$ is not the new vertex obtained from subdividing an edge, so $G[B]$ admits a feedback vertex set of size at most $t_g(m_B-1)+r_g(G[B])$ containing at least one of $u_B,v_B$ by Lemma \ref{vertex property F}, a contradiction.
\end{proof}

A splitter is {\em tight} if there exists no splitter $[A',B']$ such that $B' \subset B$.

\begin{lemma} \label{splitter 1}
If $[A,B]$ is a tight splitter, then either $G[A]=K_2$, or $\phi(G[A]) \leq t_gm_A -2t_g$.
\end{lemma}

\begin{proof}
Since $G[B]$ is 2-connected and contains at least two vertices of degree two, $r_g(G[B]) \geq 0$ by Lemma \ref{no splitter with B in F}.
So $\lvert S_B \rvert \leq t_gm_B$.
We assume that $G[A] \neq K_2$, and we shall prove that $\lvert S_A \rvert \leq t_gm_A - 2t_g$.

We first assume that $u_A$ is adjacent to $v_A$.  
Since $G[A] \neq K_2$, $u_A$ and $v_A$ can not both have degree two in $G$.
So both $u_A$ and $v_A$ have degree three in $G$, otherwise one of $u_A,v_A$ is a cut-vertex in $G$.  
Let $u'_A$ be the neighbor of $u_A$ in $G[A]$ other than $v_A$.  
Let $A' = A - \{u_A,v_A\}$.  
Since $G$ is $g$-minimal, $G[A']$ admits a feedback vertex set $S_{A'}$ with size at most $t_g(m_A - 3) + r_g(G[A'])$.  
Since $S_{A'} \cup \{u_A\} \cup S_B$ is a feedback vertex of $G$ of size at most $t_gm_A-3t_g+r_g(G[A'])+1+t_gm_B = t_gm-5t_g+1+r_g(G[A'])$.
Since $\phi(G)>t_gm$, we know $r_g(G[A'])>5t_g-1$.

If $g=4$, then $r_g(G[A']) \geq \frac{2}{9}$.
Since $G[A']$ has girth at least four and contains at least two vertices of degree two, $G[A']$ does not contain a triangle or $K_4^+$.
By Lemma \ref{r for girth 4}, $r_4(G[A'])=\frac{2}{9}$ and $G[A']$ is obtained from a disjoint union of copies of graphs in $\F_{3,1}$ by adding edges such that each new edge is a cut-edge.
By Lemma \ref{vertex property F}, $S_{A'}$ can be chosen such that $u_A' \in S_{A'}$.
Then $S_{A'}$ is a feedback vertex of $G[A]$ with size at most $\frac{2}{9}(m_A-3)+\frac{2}{9} = \frac{2}{9}m_A-\frac{4}{9} = t_gm_A-2t_g$, as claimed

If $g=5$, then $r_{t_g}(G[A']) \geq \frac{1}{5}$.
Since $G[A']$ has girth at least five and contains at least two vertices of degree two, by Lemma \ref{r for girth 5}, $\rf(G[A'])=\frac{1}{5}$ and $G[A']$ is obtained from a disjoint union of copies of graphs in $\F_{4,2}$ by adding edges such that each new edge is a cut-edge in $G[A']$.
By Lemma \ref{vertex property F}, $S_{A'}$ can be chosen such that $u_A' \in S_{A'}$.
Then $S_{A'}$ is a feedback vertex of $G[A]$ with size at most $\frac{1}{5}(m_A-3)+\frac{1}{5} = \frac{1}{5}m_A-\frac{2}{5} = t_gm_A-2t_g$, as claimed.

Hence we may assume that $u_A$ is not adjacent to $v_A$.
Let $G'' = G[A]+u_Av_A$. 
If $G[A]$ is 2-connected, then $G''$ is 2-connected.
If $G[A]$ is not 2-connected, then it has precisely two end-blocks, each containing $u_A$ and $v_A$, respectively, since $G$ is 2-connected.
So $G''$ is 2-connected in either case.  
$G''$ admits a feedback vertex set $S''$ with size at most $t_g(m_A + 1) +r_g(G'')$.  
Then $S'' \cup S_B$ is a feedback vertex set in $G$ with size at most $t_g(m-1)+r_g(G'')$ vertices.
So $r_g(G'')>t_g$.

If $g=4$, then $\rt(G'') \geq \frac{1}{3}$.
Since $G''$ is 2-connected and $G''-u_Av_A$ has girth at least four, $G'' \in \{K_3,K_4^+\} \cup \F_{2,2} \subseteq \F_{3,0} \cup \F_{3,1} \cup \F_{2,2}$ by Lemma \ref{r for girth 4}.
Since $G[A]=G''-u_Av_A$, by Lemma \ref{edge property F}, $\phi(G[A]) \leq t_4m_A+r_4(G[A])-1 \leq t_4m_A-2t_4$, where the last inequality follows from Lemma \ref{r for girth 4}.

If $g=5$, then $\rf(G'') \geq \frac{2}{5}$.
Since $G''-u_Av_A$ has girth at least five and $G''$ is 2-connected, by Lemma \ref{r for girth 5}, $\rf(G'')=\frac{2}{5}$ and $G'' \in \bigcup_{j=0}^3 \F_{3,j}$.
Since $G[A]=G''-u_Av_A$, by Lemma \ref{edge property F}, $\phi(G[A]) \leq t_5m_A+r_5(G[A])-1 \leq t_5m_A-2t_5$, where the last inequality follows from Lemma \ref{r for girth 5}.

Therefore, $\lvert S_A \rvert =\phi(G[A]) \leq  t_gm_A - 2t_g$ vertices.
\end{proof}

\begin{lemma} \label{splitter 2}
If $[A,B]$ is a tight splitter, then $G[A]=K_2$.
\end{lemma}

\begin{proof}
Suppose that $G[A] \neq K_2$.
By Lemma \ref{no splitter with B in F}, $\lvert S_B \rvert \leq t_gm_B$.
By Lemma \ref{splitter 1}, $\lvert S_A \rvert \leq t_gm_A - 2t_g$.

Let $H = G[B]-u_B$.  
Since $G$ is $g$-minimal, $H$ admits a feedback vertex set $S_H$ of size at most $t_g(m_B - 2) + r_g(H)$.  
Then $S_A\cup S_H\cup\{u_B\}$ is a feedback vertex set of $G$ of size at most
$t_gm_A - 2t_g + t_g(m_B - 2) + r_g(H) + 1 = t_gm - 6t_g + 1 + r_g(H)$.
Therefore $r_g(H) > 6t_g - 1>0$.  

Suppose $g=4$.
Then $r_g(H) > \frac{1}{3}$.
By Lemma \ref{r for girth 4}, either $H \in \{K_3,K_4,K_4^+\} \cup \F_{2,2}$, or $H$ contains $K_4^+$ as a subgraph.
Note that $H \neq K_3$ as $r_4(K_3)=\frac{1}{3}$, so either $H$ contains at most one vertex of degree at most two, or $K_4^+$ is an end-block of $H$ and $H \neq K_4^+$.
But $G[A]$ is 2-connected, every end-block of $H$ contains at least one vertex of degree at most two, and if $H$ is 2-connected, then $H$ contains at least two vertices of degree at most two, a contradiction.

Hence $g=5$ and $r_g(H) > \frac{1}{5}$.  
Note that every end-block of $H$ contains at least one vertex of degree at most two in $H$, since $G[A]$ is 2-connected.
By Lemma \ref{r for girth 5}, there exists a nonnegative integer $k$ such that $H$ is obtained from a disjoint union of a member in $\F_{3,0} \cup \F_{3,1}$ and $k$ copies of members of $\bigcup_{j=0}^2\F_{4,j}$ by adding edges between different graphs such that each new edge is a cut-edge.
But every member of $\F_{3,0} \cup \F_{3,1}$ contains a cycle of length less than five by Lemma \ref{basic F}, so the copies of the chosen members of $\bigcup_{j=0}^2 \F_{4,j}$ must have girth at least five. 
Since $[A,B]$ is tight, $k=0$.

Since $k=0$, $H \in \F_{3,0} \cup \F_{3,1}$.
But if $H \in \F_{3,1}$, then $u_B$ is the only vertex of $G[B]$ of degree at most two in $G[B]$, contradicting that $G$ is 2-connected.
Therefore, $H=K_3$ and hence $G[B]$ is the graph obtained from $K_4$ by deleting an edge.

Suppose that $u_A$ is adjacent to $v_A$.
Then $\lvert A - \{u_A,v_A\} \rvert \geq 2$, otherwise $G[A]$ is a triangle, a contradiction.
So $[A-\{u_A,v_A\},B \cup \{u_A,v_A\}]$ is a splitter.
But $G[B \cup \{u_A,v_A\}] \in \F_{3,1}$, contradicting Lemma \ref{no splitter with B in F}.
Hence $u_A$ is not adjacent to $v_A$.

Let $G'=G[A]+u_Av_A$.
Since $G[A]$ has girth at least five, $G'$ does not contain two disjoint cycles of length less than five.
So $G'$ admits a feedback vertex set $S'$ of size at most $\frac{1}{5}(m_A+1)+\rf(G')$.
Since $S' \cup S_B$ is a feedback vertex set of $G$ of size at most $\frac{1}{5}(m_A+1)+\rf(G')+\frac{1}{5}m_B = \frac{1}{5}m+\rf(G')-\frac{1}{5}$.
So $\rf(G') \geq \frac{2}{5}$.
Since $G$ is 2-connected, $G'$ is 2-connected.
Since $G'$ does not contain two disjoint cycle of length less than five, $G' \in \F_{3,j}$ for some integer $j$ with $0 \leq j \leq 3$.
Let $G''$ be the graph obtained from $G'$ by subdividing $u_Av_A$, and let $w$ be the new vertex.
Then $G = G'' \circ(wv_A,wv_A,w)$.
Hence $G \in \F_{4,j+1}$, a contradiction.
This proves the lemma.
\end{proof}

\begin{lemma} \label{lemma:contracting}
If $u$ is a vertex of degree two in $G$, then $G$ contains a cycle $C$ with length $g$ and a cycle $C'$ with length less than $g$ such that $u \in V(C)$ and $V(C) \cap V(C') = \emptyset$.
\end{lemma}

\begin{proof}
Let $G'$ be the graph obtained from $G$ by contracting an edge incident to $u$.
Note that $G'$ is simple otherwise there exists an edge-cut $[A,B]$ of order two such that $G[B]$ is a triangle containing $u$, contradicting Lemma \ref{no splitter with B in F}.
Suppose that $G'$ does not contain two disjoint cycles of length less than $g$.
Since $G$ is $g$-minimal, $G'$ admits a feedback vertex set $S$ of size at most $t_g(m-1) + r_g(G')$.  
Note that $S$ is a feedback vertex set of $G$.  
Therefore $r_g(G') - t_g > 0$.  
Since $G'$ is 2-connected, by Lemmas \ref{r for girth 4} and \ref{r for girth 5}, $G'\in\F_{i,j}$ for some integers $i,j$ with $0 \leq i \leq 4$ and $0 \leq j \leq i$.  
Then $G\in\F_{i+1,j}$, a contradiction.  
Hence $G'$ contains two disjoint cycles of length less than $g$.
Note that one of these cycles contains $u$, so $G$ contains two disjoint cycles, where one has length $g$ and the other has length less than $g$.
\end{proof}

We say a graph $H$ is {\em internally $3$-edge-connected} if $H$ is 2-edge-connected and for every edge-cut $[A,B]$ of $H$ of order two, $\lvert A \rvert=1$ or $\lvert B \rvert=1$.

\begin{lemma} \label{lemma:3-connected}
$G$ is internally 3-edge-connected.
\end{lemma}

\begin{proof}
Suppose that $G$ is not internally 3-edge-connected.
So there exists an edge-cut $[A,B]$ of $G$ of order two such that $\lvert A \rvert, \lvert B \rvert \geq 2$.
Since $G$ does not contain two disjoint cycles of length less than $g$, we may assume that $G[A]$ has girth at least $g$ by symmetry.
Subject to that, we further assume that $B$ is minimal.
Let $u_A,v_A$ be the ends of the edges between $A,B$ in $A$, and let $u_B$ and $v_B$ be the ends of the edges between $A,B$ in $B$.
Let $n_A = \lvert A \rvert$, $n_B=\lvert B \rvert$, $m_A = \lvert E(G[A]) \rvert$ and $m_B = \lvert E(G[B]) \rvert$.

We claim that $G[B]$ is 2-connected.
Suppose that $G[B]$ contains a leaf $v$.
Then $v \in \{u_B,v_B\}$.
Since $G[A]$ has girth at least $g$, $\lvert B \rvert \geq 3$ by Lemma \ref{lemma:contracting}.
Then $[A \cup \{v\},B-\{v\}]$ is an edge-cut of order two with both sides at least two vertices such that $G[A \cup \{v\}]$ has girth at least $g$, contradicting the minimality of $B$.
So $G[B]$ has minimum degree two.
Suppose that $G[B]$ is not 2-connected.
Then there exists a non-trivial end-block $B'$ of $G[B]$ such that $B'$ has girth at least $g$.
Since $G$ is 2-connected, $B'$ contains $u_B$ or $v_B$.
Then $[A \cup V(B'), B]$ is an edge-cut of $G$ of order two such that $G[A \cup V(B')]$ has girth at least $g$, contradicting the minimality of $B$.
Therefore $G[B]$ is 2-connected, as claimed.

Hence, $[A,B]$ is a splitter.
The minimality of $B$ implies that $[A,B]$ is tight.
By Lemma \ref{splitter 2}, $G[A]=K_2$.
By Lemma \ref{lemma:contracting}, there exists a cycle $C$ of length $g$ containing $u_A$ and $v_A$ and a cycle $C'$ of length less than $g$ such that $C$ and $C'$ are disjoint.
Note that $C$ is an induced cycle since otherwise $G$ contains two disjoint cycles of length less than $g$.

If $g=4$, then $u_B$ is adjacent to $v_B$.
Since $G$ is 2-connected, $G[B-\{u_B,v_B\}]$ is connected and contains at most two end-blocks.
By Lemma \ref{no splitter with B in F}, $G[B]$ is not a cycle.
So $G[B]-\{u_B,v_B\}$ contains at least two vertices and is not a path.
Then $G[B]-\{u_B,v_B\}$ has a non-trivial block $B''$ such that $G[B]-V(B'')$ does not contain a cycle of length less than $g$.
So $[V(G)-V(B''),V(B'')]$ is a splitter, contradicting the tightness of $[A,B]$.
Therefore, $g=5$.

Let $G'=G-\{u_A,v_A,v_B\}$.
Note that $v_B$ has degree three in $G$ since $G[B]$ is 2-connected.
By the minimality of $G$, $G'$ admits a feedback vertex set $S'$ of size at most $\frac{1}{5}(m-5)+\rf(G')+1=\frac{1}{5}m+\rf(G')$.
So $\rf(G')>\rf(G)=0$.
Clearly, $G' \neq K_4$ and does not contain $K_4^+$ as a subgraph.

Let $y$ be the vertex in $V(C)-\{u_A,v_A,u_B,v_B\}$.  
If $u_B$ and $y$ belong to different blocks of $G'$, then $y$ is a cut-vertex of $G'$, and hence is a cut-vertex of $G[B]$, since $v_B$ is adjacent to $y$, a contradiction.
So $u_B$ and $y$ belong to the same block of $G'$.
Since $G[B]$ is 2-connected, the block of $G'$ containing $u_B$ and $y$, denoted by $X$, is an end-block.
Hence $G'$ has at most two end-blocks, where at least one of them contains at least two vertices of degree at most two in $G'$.
By Lemma \ref{r for girth 5}, either $G'$ contains a block that is isomorphic to a member of $\bigcup_{j=0}^1\F_{3,j}$, or $G'$ is obtained from a disjoint union of copies of members of $\bigcup_{j=0}^2 \F_{4,j}$ by adding edges such that the added edges are cut-edges.

Suppose that $G'$ contains a block, denoted by $W$, isomorphic to a member of $\F_{3,0} \cup \F_{3,1}$.
If $G'$ is 2-connected, then $G'$ contains at least three vertices of degree at most two, so $G' = W = K_3$, but then $G$ has no cycle of length less than five disjoint from $C$, a contradiction.
So $G'$ is not 2-connected.
Since each member of $\F_{3,0} \cup \F_{3,1}$ contains a cycle of length less than five by Lemma \ref{basic F}, $W$ is the block of $G'$ containing a cycle of length less than five.
Since $[A,B]$ is tight, $W$ contains three vertices incident with cut-edges of $G'$, so $G'$ contains three end-blocks, a contradiction.
Therefore, $G'$ is obtained from a disjoint union of copies of members of $\bigcup_{j=0}^2 \F_{4,j}$ by adding edges such that the added edges are cut-edges.

Since $X$ contains at least two vertices of degree at most two in $G'$, and $G'$ contains at least two blocks, $X$ contains at least three vertices of degree two in $X$.
Hence $X \in \F_{4,0} \cup \F_{4,1}$, so $X$ contains a cycle of length less than five by Lemma \ref{basic F}.
Hence blocks of $G'$ other than $X$ have girth at least five and belong to $\F_{4,2}$.
By Lemma \ref{girth 5 i=4}, each block of $G'$ other than $X$ is isomorphic to $R$.
If $G'$ contains a non-trivial block other than $X$, then $G[B]$ is obtained from a graph in $\F_{4,j,k}^5$ for some $j \in \{0,1\}$ and positive integer $k$ by subdividing an edge, contradicting Lemma \ref{no splitter with B in F}.
Hence $G'=X$.
Since there exists a cycle of length less than five disjoint from $C$, $G'$ has at least five vertices, so $G' \in \F_{4,1}$.
Let $z$ be the vertex in $G'$ other than $y$ adjacent to $v_B$ in $G$.
So $u_B,y,z$ are the vertices of degree two in $G'$.
By Lemma \ref{commuting operations}, the graph, denoted by $G''$, obtained from $G'$ by suppressing $y,z$ belongs to $\F_{2,1}$.
Then $G$ can be obtained from $G''$ by taking the operation $\circ$ and then subdividing an edge twice.
Therefore, $G \in \F_{4,2}$, a contradiction.
This proves the lemma.
\end{proof}

\begin{lemma} \label{lemma:menger}
If $u$ and $v$ are vertices of $G$ of degree three, then there exist three internally disjoint paths from $u$ to $v$.
\end{lemma}

\begin{proof}
Suppose $u$ and $v$ are vertices of $G$ of degree three for which there are no three internally disjoint paths from $u$ to $v$.  
By Menger's theorem, there exist two vertices $x$ and $y$ such that $u$ and $v$ are in different components, $A$ and $B$, of $G-\{x,y\}$.  
Since $G$ is subcubic, we can assume without loss of generality that $x$ is adjacent to only one vertex $x'\in V(A)$.  
If $y$ is adjacent to only one vertex in $A$, call it $y'$.  
Otherwise let $y'$ be the vertex in $B$ adjacent to $y$.  

Then $G-\{xx', yy'\}$ contains at least two components, one containing $A$ and another containing $B$.  
Since $u$ and $v$ have degree three, these components each have at least two vertices, contradicting that $G$ is internally 3-edge-connected.
\end{proof}

\begin{lemma}\label{lemma:remove}
If $X\subseteq V(G)$ and $\lvert E(X, V(G-X)) \rvert = 3$, then $G-X$ is connected and has at most one nontrivial block.  
Furthermore, if $u,v\in V(G-X)$ are adjacent to vertices in $X$ but $uv \not \in E(G)$, then $G- X + uv$ has exactly one nontrivial block, and this block contains $u$ and $v$.
\end{lemma}

\begin{proof}
Note that $G-X$ is connected, since otherwise $G$ has a cut-edge.  
Suppose $B_1$ and $B_2$ are distinct nontrivial blocks in $G- X$.  
Since $G-X$ is connected, there are vertices $a\in V(B_1)$ and $b\in V(B_2)$ of degree three in $G$.  
By Lemma \ref{lemma:menger}, there are three internally disjoint paths from $a$ to $b$ in $G$.  
Since $\lvert E(X, G-X) \rvert=3$, at most one of these paths can contain vertices of $X$.
Then the other two form a cycle in $G-X$, so $a$ and $b$ are in the same block, a contradiction.

Suppose $u,v\in V(G- X)$ are adjacent to vertices in $X$ but $uv \not \in E(G)$.
Since $G-X$ is connected, $u$ and $v$ are in the same nontrivial block, $B$, in $(G- X)+uv$.  
Suppose $B'$ is another nontrivial block in $(G-X)+uv$.  
Since $(G- X)+uv$ is connected, there are vertices $a\in V(B)$ and $b\in V(B')$ of degree three in $G$.  
Then by Lemma \ref{lemma:menger}, there are three internally disjoint paths from $a$ to $b$ in $G$.  
Since $\lvert E(X, V(G- X)) \rvert=3$, at most one of these paths can contain vertices of $X$.  
Then the other two form a cycle in $G-X$, so $a$ and $b$ are in the same block, a contradiction.
\end{proof}

\begin{lemma} \label{lemma:notriangle}
$G$ has no triangle.
\end{lemma}

\begin{proof}
Suppose $a,b,$ and $c$ are the vertices of a triangle in $G$.  
Since $G$ is not the triangle and is not the graph obtained from $K_4$ by deleting an edge, $a,b,$ and $c$ have degree three by Lemma \ref{lemma:3-connected}.
Let $X = \{a,b,c\}$.  
Let $a',b',$ and $c'$ be the neighbors of $a,b,$ and $c$ not in $X$ respectively.
If $a',b',$ and $c'$ are not pairwise distinct, then $G \in \{K_4,K_4^+\}$ by Lemma \ref{lemma:3-connected}.
Therefore $a',b'$, and $c'$ are pairwise distinct.
If $a',b',$ and $c'$ are pairwise adjacent, then $G$ contains two disjoint triangles, a contradiction.
Suppose one of $a',b',$ and $c'$ is adjacent to the other two.  We may assume without loss of generality that $a'$ is adjacent to $b'$ and $c'$.  By Lemma \ref{lemma:3-connected}, $G - X - \{a',b',c'\}$ is either the empty graph or an isolated vertex.  
In the first case, $G\in \F_{3,1}$, a contradiction.   
In the second case, $\phi(G) =2 \leq t_gm$, a contradiction.

Therefore, $|E(G[\{a',b',c'\}]) | \leq 1$.
By symmetry, we may assume that $b'$ is not adjacent to $a'$ or $c'$.
By Lemma \ref{lemma:remove}, $(G- X)+b'c'$ is connected and has at most one nontrivial block, $B$.  
Since $G- X$ has girth at least $g$, $B$ admits a feedback vertex set $S$ of size at most $t_g(m - 5) + \epsilon_g(B) = t_gm - 5t_g + \epsilon_g(B)$.
Then $S\cup\{a\}$ is a feedback vertex set of $G$ of size at most $t_gm + \epsilon_g(B) + 1 - 5t_g$, so $\epsilon_g(B) > 5t_g - 1$.  
Therefore $\epsilon_g(B) > 0$. 

Since $(G-X)+b'c'$ has only one nontrivial block, by Lemmas \ref{r for girth 4} and \ref{r for girth 5}, either $(G-X)+b'c' \in \F_{i,j}$ for some integers $i,j$ with $1 \leq i \leq 4$ and $0 \leq j \leq i$, or $(G-X)+b'c'$ contains $K_4^+$ as a subgraph.
For the former, $G = ((G-X)+b'c') \circ (b'c',b'c',a)$, so $G \in \F_{i,j+1}$, a contradiction.
For the latter, since $G$ is 2-connected, $(G-X)+b'c'$ can be obtained from $K_4^+$ by attaching a leaf, so $G \in \F_{3,2}$, a contradiction.
\end{proof}

\begin{lemma} \label{cubic in short cycle}
No vertex in a cycle of $G$ of length less than five has degree two in $G$.
\end{lemma}

\begin{proof}
Let $D$ be a cycle of length less than five containing a vertex $v$ of degree two in $G$.
$D$ has length four by Lemma \ref{lemma:notriangle}.
By Lemma \ref{lemma:contracting}, there exists a cycle $C$ of length $g$ containing $v$ and a cycle $C'$ with length less than $g$ disjoint from $C$.
Since $v$ has degree two, $D$ shares at least three vertices with $C$.
Since $G$ is subcubic, $C'$ and $D$ are disjoint cycles of length less than $g$, a contradiction.
\end{proof}

\begin{lemma}\label{lemma:op2makesbad}
Let $a$, $b$ and $c$ be distinct vertices of degree three in $G$ such that $abc$ is a path in $G$.  
Let $a_1$ and $a_2$ be the neighbors of $a$ other than $b$, let $c_1$ and $c_2$ be the neighbors of $c$ other than $b$, and let $b'$ be the other neighbor of $b$.  
If $G'= (G-\{a,b,c\})+a_1a_2+c_1c_2$, then $G'$ contains two disjoint cycles of length less than $g$.
\end{lemma}

\begin{proof}
Note that $a_1a_2,c_1c_2 \not \in E(G)$ since $G$ has no triangle.
By Lemma \ref{lemma:3-connected}, $G'$ is connected.  
Note that $G=G' \circ (a_1a_2,c_1c_2,b')$, so $G' \not \in \bigcup_{i \geq 1, j \geq 0} \F_{i,j}$.
If $G'$ has at least two nontrivial blocks, then $G-b$ has at least two nontrivial blocks.
But $\lvert E(\{b\}, V(G- b)) \rvert = 3$, so $G-b$ has at most one nontrivial block by Lemma \ref{lemma:remove}.  
Therefore, $G'$ has exactly one nontrivial block $B$.

Suppose that $G'$ does not contain two disjoint cycles of length less than $g$.
Since $G$ is $g$-minimal, $G'$ admits a feedback vertex set $S$ of size at most 
$t_g(m - 5) + r_g(G') = t_gm + r_g(G') - 5t_g$.
But $S\cup\{b\}$ is a feedback vertex set of $G$ of size at most $t_gm + r_g(G') - 5t_g + 1$.  
Therefore $r_g(G') > 5t_g - 1 \geq 0$.

Since $G' \not \in \bigcup_{i \geq 1 ,j \geq 1} \F_{i,j}$, and $B$ is the unique nontrivial block of $G'$, and $G'$ has no triangle, by Lemmas \ref{r for girth 4} and \ref{r for girth 5}, $G'$ can be obtained from a graph in $\{K_4^+\} \cup \bigcup_{j=1}^2\F_{3,j} \cup \bigcup_{j=0}^3 \F_{4,j} \subseteq \bigcup_{i=2}^4 \bigcup_{j=0}^{i-1} \F_{i,j}$ by attaching trees.
But every vertex of $G'$ other than $b'$ has degree at least two in $G'$, so we can attach at most one tree, and this tree is a path $P$.
Hence there exist integers $i,j$ with $2 \leq i \leq 4$ and $0 \leq j \leq i-1$ such that $G \in \F_{i+\lvert E(P) \rvert, j+1}$, a contradiction.
\end{proof}

\begin{lemma} \label{K_{2,3} neighbors}
If $D$ is a subgraph of $G$ isomorphic to $K_{2,3}$, then there are three distinct vertices in $G-V(D)$ adjacent in $G$ to vertices in $D$.
\end{lemma}

\begin{proof}
Let $D$ be a subgraph of $G$ isomorphic to $K_{2,3}$.
Since $G$ has no triangle, $D$ is an induced subgraph of $G$.
By Lemma \ref{cubic in short cycle}, every vertex in $D$ has degree three.
Let $a,b,c$ be the vertices of degree two in $D$.
Let $a',b',c'$ be the neighbors of $a,b,c$ not contained in $D$, respectively.
Since $K_{3,3} \in \F_{3,1,0}^4 \cap \F_{3,1,0}^5$, $G \neq K_{3,3}$, so not all of $a',b',c'$ are the same.

Suppose that $a',b'$, and $c'$ are not pairwise distinct.  
We may assume without loss of generality that $a'=b'$.  
Since $[V(D) \cup \{a'\}, V(G)-V(D)-\{a'\}]$ is an edge-cut of order two, by Lemma \ref{lemma:3-connected}, $G-V(D)-\{a'\}$ is an isolated vertex, so $c'$ is adjacent to $a'$.  
Then $|E(G)| = 10$, and $\{a,c\}$ is a feedback vertex set of $G$, a contradiction.
Hence $a',b',c'$ are three distinct vertices in $G-V(D)$ adjacent in $G$ to vertices in $D$.
\end{proof}

\begin{lemma} \label{girth 5}
$G$ has girth at least five.
\end{lemma}

\begin{proof}
Suppose that $C$ is a cycle of length less than five.
By Lemma \ref{lemma:notriangle}, $C$ has length four.  
Say $C=abcda$.
By Lemma \ref{cubic in short cycle}, $a,b,c,$ and $d$ have neighbors $a',b',c',$ and $d'$ not in $C$, respectively.
By Lemma \ref{K_{2,3} neighbors}, $a' \neq c'$ or $b' \neq d'$.
By symmetry, we may assume without loss of generality that $a' \neq c'$.

Let $G' = (G-\{a,b,c\})+\{a'd, c'd\}$.
By Lemma \ref{lemma:op2makesbad}, $G'$ contains two disjoint cycles $D_1$ and $D_2$ of length less than $g$.
Since $G$ does not contain two such cycles, one of them, say $D_1$, contains an edge in $\{a'd,c'd\}$.
In particular, $d \in V(D_1)$, so $D_2$ contains none of $a'd,c',$ and $d$.
Since $D_2$ is in $G'$ and is disjoint from $D_1$, $V(D_2)\cap V(C) = \emptyset$.
Therefore $C$ and $D_2$ are disjoint cycles of length less than five in $G$.
Hence $g=4$.  Then $D_2$ is a triangle in $G$, contradicting Lemma \ref{lemma:notriangle}.
This proves the lemma.
\end{proof}

\begin{lemma} \label{g=5}
$G$ is cubic and 3-connected, and $g=5$.
\end{lemma}

\begin{proof}
Lemmas \ref{lemma:contracting} and \ref{girth 5} imply that $G$ is cubic.
By Lemma \ref{lemma:menger}, $G$ is 3-connected.

Suppose that $g=4$.
Let $abc$ be a path in $G$ on three vertices.
Let $a_1,a_2$ be the neighbors of $a$ other than $b$, and let $c_1,c_2$ be the neighbors of $c$ other than $b$.
Since $G$ has girth at least five, $\{a_1,a_2\} \cap \{c_1,c_2\} = \emptyset$.
Let $G' = (G-\{a,b,c\})+a_1a_2+c_1c_2$.
By Lemma \ref{lemma:op2makesbad}, $G'$ contains two disjoint triangles.
Since $\{a_1,a_2\} \cap \{c_1,c_2\} = \emptyset$, each triangle contains one of $a_1a_2$ and $c_1c_2$.
Hence $G$ contains a cycle of length four, a contradiction.
Therefore, $g=5$.
\end{proof}

\begin{lemma} \label{lemma:vertexadjacentto5-cycles}
Let $v\in V(G)$, and let $a,b,$ and $c$ be the neighbors of $v$.  
Let $a_1$ and $a_2$ be the neighbors of $a$ other than $v$, and let $b_1$ and $b_2$ be the neighbors of $b$ other than $v$.  
Then $G-v$ contains two disjoint 5-cycles, where one contains the path $a_1aa_2$ and the other contains the path $b_1bb_2$.
\end{lemma}

\begin{proof}
Let $G' = (G-\{a,v,b\})+a_1a_2+b_1b_2$.
By Lemma \ref{lemma:op2makesbad}, $G'$ contains two disjoint cycles $D_1,D_2$ of length less than five.
Since $G$ has girth at least five, each $D_1,D_2$ contains exactly one of $a_1a_2$ and $b_1b_2$.
Say $D_1$ contains $a_1a_2$ and $D_2$ contains $b_1b_2$.
So $G[V(D_1) \cup \{a\}]$ and $G[V(D_2) \cup \{b\}]$ are two disjoint cycles in $G-v$ of length less than six.
These two cycles have length five since $G$ has girth five.
\end{proof} 

By Theorem \ref{characterization dodecahedron}, $G$ is the dodecahedron.
However, the dodecahedron has 30 edges and a feedback vertex set of size 6, a contradiction.  
This completes the proof of Theorem \ref{real main thm}.

\bigskip

\noindent{\bf Acknowledgement:} The authors thank Michael Genter for bringing \cite{gr} to our attention and thank the anonymous reviewer for careful reading of this paper.


\begin{thebibliography}{99}	
\bibitem{abt}
M. Albertson, B. Bollobas and S. Tucker, {\em The independence ratio and maximum degree of a graph}, Proc. 7th Southeastern Conf. on Combinatorics, Graph Theory, and Computing, LSU (1976), 43--50.

\bibitem{AMT01}
N. Alon, D. Mubayi, and R. Thomas,
\newblock {\em Large induced forests in sparse graphs,}
\newblock J. Graph Theory 38 (2001), 113--123.

\bibitem{BHS87}
J.~A. Bondy, G. Hopkins, and W. Staton,
\newblock {\em Lower bounds for induced forests in cubic graphs,}
\newblock Canad. Math. Bull. 30 (1987), 193--199.

\bibitem{bccs}
F. C. Bussemaker, S. \v{C}obelji\'{c}, D. M. Cvetkovi\'{c} and J. J. Seidel, {\em Computer investigation of cubic graphs}, T.H. Report 76-WSK-01, Department of Mathematics, Technological University Eindhoven, The Netherlands.

\bibitem{f}
S. Fajtlowicz, {\em The indepenedence ratio for cubic graphs}, Proc. 8th Southeastern Conf. on Combinatorics, Graph Theory, and Computing, LSU (1977), 273--277.

\bibitem{gr}
M. Gentner and D. Rautenbach, {\em Feedback vertex sets in cubic multigraphs}, Discrete Math. 338 (2015), 2179--2185.

\bibitem{k}
R. M. Karp, {\em Reducibility among combinatorial problems}, Complexity of computer computations, Proc. Sympos., IBM Thomas J. Watson Res. Center, Yorktown Heights, N.Y., (1972), 85--103.

\bibitem{kl}
T. Kelly and C.-H. Liu, {\em Minimum size of feedback vertex sets of planar graphs of girth at least five}, European J. Combin. 61 (2017), 138--150.

\bibitem{KLS10}
{\L}. Kowalik, B. Lu{\v{z}}ar and R. {\v{S}}krekovski, {\it An improved bound on the largest induced forests for triangle-free planar graphs}, Discrete Math. Theor. Comput. Sci. 12 (2010), 87--100.

\bibitem{ly}
C.-H. Liu and G. Yu, {\em Linear colorings of subcubic graphs}, European J. Combin. 34 (2013), 1040--1050.

\bibitem{LZ96}
J. Liu and C. Zhao,
\newblock {\em A new bound on the feedback vertex sets in cubic graphs,}
\newblock Discrete Math., 148 (1996), 119--131.

\bibitem{s}
W. Staton, {\em Some Ramsey-type numbers and the independence ratio}, Trans. Amer. Math. Soc. 256 (1979), 353--370.

\bibitem{ukg}
S. Ueno, Y. Kajitani and S. Gotoh, {\em On the nonseparating independent set problem and feedback set problem for graphs with no vertex degree exceeding three}, Discrete Math., 72 (1988), 355-360.

\bibitem{ZL90}
M.~L. Zheng and X. Lu,
\newblock {\em On the maximum induced forests of a connected cubic graph without
  triangles,}
\newblock Discrete Math., 85 (1990), 89--96.

\bibitem{z}
X. Zhu, {\em Bipartite subgraphs of triangle-free subcubic graphs}, J. Combin. Theory, Ser. B 99 (2009), 62--83.
\end{thebibliography}

\end{document}